\documentclass[a4paper,twoside,final]{siamltex}
\usepackage{graphicx}
\usepackage{epstopdf}
\usepackage{amssymb,amsmath}
\usepackage{epsfig}
\usepackage{amsfonts}
\usepackage{latexsym}
\usepackage{graphicx}
\usepackage{subfigure}
\usepackage{color}
\usepackage{url}
\usepackage{framed}
\usepackage{srcltx} 

\def\eqref#1{{\rm (\ref{#1})}}
\newtheorem{remark}{Remark}[section]

\newcommand{\weak}{\rightharpoonup}
\newcommand{\E}{D}
\renewcommand{\k}{\kappa}

\newcommand{\escu}[2]{\left\langle #1\,,#2 \right\rangle}

\begin{document}
\title{Image denoising: learning noise distribution via PDE-constrained optimization}
\author{Juan Carlos De los Reyes\footnotemark[1] \and Carola-Bibiane Sch\"onlieb\footnotemark[2]}
\renewcommand{\thefootnote}{\fnsymbol{footnote}}

\footnotetext[1]{Departamento de Matem\'atica, Escuela Polit\'ecnica Nacional
de Quito, Ecuador (juan.delosreyes@epn.edu.ec)} 
\footnotetext[2]{Department of Applied Mathematics and Theoretical Physics (DAMTP), University of Cambridge (C.B.Schoenlieb@damtp.cam.ac.uk)}
\footnotetext[3]{Research partially supported by the Alexander von Humboldt Foundation. Moreover, CBS acknowledges the financial support provided by the Cambridge Centre for Analysis (CCA) and the Royal Society International Exchanges Award IE110314 for the project \emph{High-order Compressed Sensing for Medical Imaging}. Further, this publication is based on work supported by Award No. KUK-I1-007-43 , made by King Abdullah University of Science and Technology (KAUST).}
\footnotetext{\emph{Date: 14. July 2012}}
\renewcommand{\thefootnote}{\arabic{footnote}}
\maketitle

\begin{abstract}
We propose a PDE-constrained optimization approach for the determination of
noise distribution in total variation (TV) image denoising. An
optimization problem for the determination of the weights
correspondent to different types of noise distributions is stated and existence of an optimal solution is
proved. A tailored regularization approach for the approximation of the optimal
parameter values is proposed thereafter and its consistency
studied. Additionally, the differentiability of the solution operator
is proved and an optimality system characterizing the optimal
solutions of each regularized problem is derived. The optimal parameter values are numerically computed by using a
quasi-Newton method, together with semismooth Newton type algorithms for the solution of the TV-subproblems.
\end{abstract}


\section{Introduction}
Let $f\in L^p(\Omega)$, $p=1$ or $2$ with $\Omega\subset \mathbb{R}^2$, be a given noisy image. Depending on the application at hand the type of noise, i.e., the noise distribution, changes \cite{Bo00}. Examples for noise distributions are Gaussian noise, which typically appears in, e.g. MRI (Magnetic Resonance Tomography),  Poisson noise in, e.g. radar measurements or PET (Positron Emission Tomography), and impulse noise usually due to transmission errors or malfunctioning pixel elements in camera sensors. 
To remove the noise a total variation (TV) regularization is frequently considered \cite{AubVes,Cha,ChaLio,chcacrnopo10,DobVog,Vese} that amounts to reconstruct a denoised version $u$ of $f$ as a minimiser of the generic functional 
\begin{equation}\label{tvbasic}
\mathcal J(u) =   |Du|(\Omega)  + \lambda \phi(u,f),
\end{equation}
with 
\begin{equation}\label{totalvariation}
|Du|(\Omega)= \sup_{{\bf g} \in C_0^\infty(\Omega;\mathbb R^2), \|g\|_\infty \leq 1} \int_\Omega  u ~\nabla \cdot {\bf g} ~dx
\end{equation} 
the total variation of $u$ in $\Omega$, $\lambda$ a positive parameter and $\phi$ a suitable distance function called the data fidelity term. The latter depends on the statistics of the data $f$, which can be either estimated or approximated by a noise model known from the physics behind the acquisition of $f$. For normally distributed $f$, i.e. the interferences in $f$ are Gaussian noise, this distance function is the squared $L^2$ norm of $u-f$. 
If a Poisson noise distribution is present, $\phi(u,f) = \int_\Omega \lambda ~(u-f\log u)~ dx$, which corresponds to the Kullback-Leibler distance between $u$ and $f$ \cite{LCA07,SBMB09}. In the presence of impulse noise, the correct data fidelity term turns out to be the $L^1$ norm of $u-f$ \cite{Nik04,DAG09}. Other noise models have been considered as well, cf. e.g. \cite{AuAu08}. The size of the parameter $\lambda$ depends on the strength of the noise, i.e. it models the trade-off between regularisation and fidelity to the measured datum $f$.

A key issue in total variation denoising is an adequate choice of the correct noise model, i.e. the choice of $\phi$, and of the size of the parameter $\lambda$. Depending on this choice, different results are obtained. The term $\phi$ is usually modelled from the physics behind the acquisition process. Several strategies, both heuristic and statistically grounded, have been considered for choosing the weight $\lambda$, cf. e.g. \cite{Cha,FMM12a,FMM12b,FMM12c,HDR11,SAC05}. In this paper we propose an optimal control strategy for choosing both $\phi$ and $\lambda$. To do so we extend model \eqref{tvbasic} to 
a more general model, that allows for mixed noise distributions in the data. Namely, instead of \eqref{tvbasic} we consider
\begin{equation}\label{general denoise}
\min_u {\left( |Du|(\Omega) + \sum_{i=1}^d  \int_\Omega \lambda_i~ \phi_i (u,f)\; dx \right)}.
\end{equation}
where $\phi_i$, $i=1,\dots,d,$ are convex differentiable functions in $u$, and $\lambda_i$ are positive parameters. The functions $\phi_i$ model different choices of data fidelities. In the case of mixed Gaussian and impulse noise $d=2$, $\phi_1(u,f)=\|u-f\|_{L^2(\Omega)}^2$ and $\phi_2(u,f) = \|u-f\|_{L^1(\Omega)}$. The parameters $\lambda_i$ weight the different noise models $\phi_i$ and the regularising term against each other. As such, the choice of these parameters depends on the amount and strength of noise of different distributions in $f$. Typically, the $\lambda_i$ are chosen to be real parameters. However, in some applications, it may be more favourable to choose them to be spatially dependent functions $\lambda_i:\Omega\rightarrow \mathbb{R}^+$, cf. e.g. \cite{ABCH08,BCRS03,FMM12b,HDR11,SAC05}.

We propose a PDE-constrained optimization approach to determine the weights $\lambda_i$ of the noise distribution and, in that manner, learn the noise distribution present in the measured datum $f$ for both $d=1$ and mixed noise models $d>1$. To do so, we treat \eqref{general denoise} as a constraint and state an optimization problem governed by \eqref{general denoise} for the optimal determination of weights. When possible, we replace the optimization problem by a necessary and sufficient optimality condition (in form of a variational inequality (VI)) as a constraint.

Schematically, we proceed in the following way:
\begin{framed}
\begin{enumerate}
\item We consider a training set of pairs $(f_k,u_k)$, $k=1,2,\ldots, N$. Here, $f_k$'s are noisy images, which have been measured with a fixed device with fixed settings, and the images $u_k$ represent the ground truth or images that approximate the ground truth within a desirable tolerance. 
\item We determine the optimal choice of functions $\lambda_i$ by solving the following problem for $k=1,2,\ldots, N$
\begin{equation}\label{optimalcontrolbasic}
\min_{\lambda_i \geq 0, ~{i=1,...,d}} ~ \| \tilde{u} - u_k\|_{L^2(\Omega)}^2 + \beta \sum_{i=1}^d \|\lambda_i\|_{X}^2,
\end{equation}
where $\tilde{u}$ solves the minimization problem \eqref{general denoise} for a given $f_k$, $X$ corresponds to $\mathbb R$ in the case of scalar parameters or to, e.g., $L^2(\Omega)$ in the case of distributed functions, and $0<\beta \ll1$ is a given weight.
\end{enumerate}
\end{framed}
The reasonability of assuming to have a such a training set is motivated by certain applications, where the accuracy and as such the noise level in the measurements can be tuned to a certain extent. In MRI or PET, for example, the accuracy of the measurements depends on the setup of the experiment, e.g., the acquisition time. Hence, such a training set can be provided by a series of measurements using phantoms. Then, the $u_k$'s are measured with the maximal accuracy practically possible and the $f_k$'s are measured within a usual clinical setup. For instance, dictionary based image reconstruction methods are already used in the medical imaging community. There, good quality measurements or template shapes are used as priors for reconstructing $\tilde u$, cf. e.g. \cite{TJFVD10}, or for image segmentation, cf. e.g. \cite{SC09} and references therein. 

Up to our knowledge this paper is the first one to approach the estimation of the noise distribution as an optimal control problem. By incorporating more than one $\phi_i$ into the model \eqref{general denoise} our approach automatically chooses the correct one(s) through an optimal choice of the weights $\lambda_i$ in terms of \eqref{optimalcontrolbasic}.

\paragraph*{Organisation of the paper:} We continue with the analysis of the optimization problem \eqref{eq:optimization problem1}--\eqref{Jminprob} in Section \ref{sec:bvanalysis}. After proving existence of an optimal solution and convergence of the Huber-regularized minimisers to a minimiser of the original total variation problem, the optimization problem is transferred to a Hilbert space setting where the rest of our analysis takes place in Section \ref{sec:h1analysis}. This further smoothing of the regularizer turns out to be necessary in order to prove continuity of the solution map in a strong enough topology and to verify convergence of our procedure. Moreover, differentiability of the regularized solution operator is thereafter proved, which leads to a first order optimality system characterization of the regularized minimisers. The paper ends with three detailed numerical experiments where the suitability of our approach is computationally verified.

\section{Optimization problem in $BV(\Omega)$}\label{sec:bvanalysis}
We are interested in the solution of the following bilevel optimization problem
\begin{subequations} \label{eq:optimizationproblem:org}
\begin{equation}
\min_{\lambda_i \geq 0, ~{i=1,...,d}} ~ g(\tilde{u}) + \beta \sum_{i=1}^d \|\lambda_i\|_{X}^2
\end{equation}
subject to
\begin{multline}
\tilde{u} = \mathrm{argmin}_{u\in BV\cap \mathcal A} \left\{\mathcal J (u) =  |Du|(\Omega) + \sum_{i=1}^d \int_\Omega \lambda_i \phi_i(u,f)~ dx \right\},\label{Jminprob:org}
\end{multline}
\end{subequations}
where the space $X$ corresponds to $\mathbb R$ in the case of scalar parameters or to a function space such that $X \hookrightarrow L^2(\Omega)$ (where $\hookrightarrow$ stands for continuous injection) in the case of distributed functions, $g: L^2(\Omega) \mapsto \mathbb R$ is a $C^1$ functional to be minimised and $\beta>0$. The admissible set of functions $\mathcal A$ is chosen according to the data fidelities $\phi_i$. In particular, $BV(\Omega) \cap \mathcal A$ restricts the set of $BV$ functions on $\Omega$ to those for which the $\phi_i$'s are well defined, cf. examples below. Moreover, we assume that the functions $\phi_i$ are differentiable and convex in $u$, are bounded from below, and fulfil the following coercivity assumption
\begin{equation}\label{phicoercive}
\int_\Omega \phi_i(u,f) ~ dx\geq C_1 \|u\|_{L^p}^p - C_2,\quad \forall u\in L^p(\Omega)\cap \mathcal A
\end{equation}
for nonnegative constants $C_1,C_2$ and at least one $p=1$ or $p=2$. Examples of $\phi_i$'s that fulfill these assumptions and that are considered in the paper are
\begin{itemize}
\item The Gaussian noise model, where $\int_\Omega \phi(u,f) ~dx= \|u-f\|_{L^2(\Omega)}^2$ fulfills the coercivity constraint for $p=2$ and the admissible set $\mathcal A=L^2(\Omega)$.
\item The Poisson noise model, where $\phi(u,f) =u-f\log u$ and $\mathcal A=\{u\in L^1(\Omega)| ~ u\geq 0\}$. This $\phi$ is convex and differentiable and fulfils the coercivity condition for $p=1$. More precisely, we have for $u\geq 0$
$$
\int_\Omega (u-f\log u)~ dx \geq \|u\|_{L^1(\Omega)} - \|f\|_{L^\infty(\Omega)}\cdot \log \|u\|_{L^1(\Omega)},
$$
where we have used Jensen's inequality, i.e., for $u\geq 0$
$$
\log \left(\int_\Omega u~ dx \right) \geq \int_\Omega \log u~ dx.
$$
\item The impulse noise model, where $\int_\Omega \phi(u,f) ~dx = \|u-f\|_{L^1(\Omega)}$ fulfills the coercivity constraint for $p=1$.
\end{itemize}

For the numerical solution of \eqref{general denoise} we want to use derivative-based iterative methods. To do so, the gradient of the total variation denoising model has to be uniquely defined. That is, a minimiser of \eqref{general denoise} is uniquely characterised by the solution of the corresponding Euler-Lagrange equation. Since the total variation regulariser is not differentiable but its ''derivative'' can be only characterised by a set of subgradients (the subdifferential), we (from now on) shall use a regularised version of the total variation. More precisely, we consider for $\gamma\gg 1$ the Huber-type regularisation of the total variation with
\begin{equation}\label{eq:huber}
|\nabla u|_\gamma = \begin{cases}
|\nabla u|-\frac{1}{2\gamma} & \textrm{if } |\nabla u|\geq \frac{1}{\gamma}\\
|\nabla u|^2 \frac{\gamma}{2} & \textrm{if } |\nabla u| < \frac{1}{\gamma}
\end{cases}
\end{equation}
and the following regularised version of \eqref{eq:optimizationproblem:org}-\eqref{Jminprob:org}
\begin{subequations} \label{eq:optimization problem1}
\begin{equation}
\min_{\lambda_i \geq 0, ~{i=1,...,d}} ~ g(\tilde{u}) + \beta \sum_{i=1}^d \|\lambda_i\|_{X}^2
\end{equation}
subject to
\begin{multline}
\tilde{u} = \mathrm{argmin}_{u\in W^{1,1}\cap \mathcal A} \left\{\mathcal J^\gamma(u) =  \int_\Omega |\nabla u|_\gamma~ dx + \sum_{i=1}^d \int_\Omega \lambda_i \phi_i(u,f)~ dx \right\},\label{Jminprob}
\end{multline}
\end{subequations}
where the space $X$, $g$, $\phi_i$'s and $\beta>0$ are defined as before. The admissible set of functions $\mathcal A$ is assumed to be convex and closed subset of $W^{1,1}(\Omega)$ and is chosen according to the data fidelities $\phi_i$, cf. examples above.  The existence of an optimal solution for \eqref{Jminprob} is proven by the method of relaxation. To do so we extend the definition of $\mathcal J^\gamma$ to $BV(\Omega)$ as
$$
\mathcal J^\gamma_{ext}(u) = \begin{cases}
\mathcal J^\gamma(u) & u\in W^{1,1}(\Omega)\cap \mathcal A\\
+\infty & u\in BV(\Omega)\setminus (W^{1,1}\cap \mathcal A)
\end{cases}
$$
and prove the existence of a minimiser for the lower-semicontinuous envelope of $\mathcal J^\gamma_{ext}$ as follows. We have the following existence result.
\begin{theorem}\label{Jrelaxexist}
Let $f\in L^2(\Omega)$ and $\lambda_i\geq 0$ fixed. Then there exists a unique solution $u\in BV(\Omega)\cap \mathcal A$ of the minimisation problem
$$
\min_{u\in BV(\Omega)\cap \mathcal A} \mathcal J^\gamma_{relax}(u),
$$
where
\begin{equation}\label{Jrelaxedfunc}
\mathcal J^\gamma_{relax}(u) = \int_\Omega |\nabla u|_\gamma~ dx + C\int_\Omega |D_s u| + \sum_{i=1}^d \int_\Omega \lambda_i \phi_i(u,f)~ dx.
\end{equation}
is the relaxed functional of $\mathcal J^\gamma_{ext}$ on $BV-w*$. 
\end{theorem}

\begin{remark}
Note that 
$$
\mathcal J^\gamma_{relax}(u) \leq \mathcal J^\gamma_{ext}(u), \quad u\in BV(\Omega)
$$
and $\mathcal J^\gamma_{relax}(u)=\mathcal J^\gamma_{ext}(u)$ for $u\in W^{1,1}(\Omega)\cap \mathcal A$. Moreover, the relaxation result from Theorem \ref{Jrelaxexist} means that
$$
\mathcal J^\gamma_{relax}(u) = \inf\left\{\liminf_n \mathcal J^\gamma_{ext}(u_n):\, u_n\in BV(\Omega),\, u_n\rightarrow u \textrm{ in } BV-w* \right\},
$$
i.e. $\mathcal J^\gamma_{relax}$ is the greatest $BV-w*$ lower semicontinuous functional less than or equal to $\mathcal J^\gamma_{ext}$.
\end{remark}

\begin{proof}
Let $u_n$ be a minimising sequence for $\mathcal J^\gamma_{relax}$. We start by stating the fact that $|\cdot|_\gamma$ is coercive and at most linear. That is
\begin{align*}
& \textrm{For } |x|\geq \frac{1}{\gamma}: \; A |x|-B \leq |x|_\gamma = |x|-\frac{1}{2\gamma} \leq |x| + 1\\
& \textrm{For } |x|<\frac{1}{\gamma}<1: \; A |x|-B \leq  = |x|_\gamma = |x|^2\frac{\gamma}{2} < |x|\frac{\gamma}{2}.
\end{align*}
Hence,
\begin{multline*}
|Du_n|(\Omega) = \int_\Omega |\nabla u_n| ~ dx + \int_\Omega |D_s u| \leq  \int_\Omega |\nabla u|_\gamma~ dx + C\int_\Omega |D_s u| \leq M, \quad \forall n\geq 1.
\end{multline*}
Moreover, $u_n$ is uniformly bounded in $L^p(\Omega)$ for $p=1$ or $p=2$ because of the coercivity assumption \eqref{phicoercive} on $\phi_i$ and therefore $u_n$ is uniformly bounded in $BV(\Omega)$. Because $BV(\Omega)$ can be compactly embedded in $L^1(\Omega)$ this gives that $u_n$ converges weak $*$ to a function $u$ in $BV(\Omega)$ and (by passing to another subsequence) strongly converges in $L^1(\Omega)$. From the convergence in $L^1(\Omega)$, $\Omega$ bounded, we get that $u_n$ (up to a subsequence) converges pointwise a.e. in $\Omega$. Moreover, since $\phi_i$ is continuous, also $\phi_i(u_n,f)$ converges pointwise to $\phi_i(u,f)$. 
Then, lower-semicontinuity of $R(|Du|) = \int_\Omega |\nabla u|_\gamma~ dx + C\int_\Omega |D_s u|$ w.r.t. strong convergence in $L^1$ \cite{DeTe84} and Fatou's lemma together with pointwise convergence applied to $\int_\Omega \phi_i(u_n,f)~dx$ gives that
\begin{multline*}
\mathcal J^\gamma_{relax}(u) =  \int_\Omega |\nabla u|_\gamma~ dx + C\int_\Omega |D_s u| + \sum_{i=1}^d \int_\Omega \lambda_i \phi_i(u,f)~ dx \\
\leq \liminf_n \mathcal J^\gamma_{relax}(u_n) = \liminf_n \left(\int_\Omega |\nabla u_n|_\gamma~ dx + C\int_\Omega |D_s u_n| + \sum_{i=1}^d \int_\Omega \lambda_i \phi_i(u_n,f)~ dx \right).
\end{multline*}
To see that the minimiser lies in the admissible set $\mathcal A$ it is enough to observe that the set $\mathcal A$ is a convex and closed subset of $BV(\Omega)$ and hence it is weakly closed by Mazur's Theorem. This gives that $u\in\mathcal A$.
To see that in fact $\mathcal J^\gamma_{relax}$ is the greatest lower-semicontinuous envelope of $\mathcal J^\gamma_{ext}$ see \cite{DeTe84,BB90,BB92,BB93,BBB95}.
\end{proof}

\begin{theorem}
There exists an optimal solution to 
\begin{subequations} \label{eq:optimization problem2}
\begin{equation}
\min_{\lambda_i \geq 0, ~{i=1,...,d}} ~ g(\tilde{u}) + \beta \sum_{i=1}^d \|\lambda_i\|_{X}^2
\end{equation}
subject to
\begin{equation}
\tilde{u} = \mathrm{argmin}_{u\in BV(\Omega)\cap \mathcal A} \mathcal J^\gamma_{relax}(u).
\end{equation}
\end{subequations}
\end{theorem}

\begin{proof}
Since the cost functional is bounded from below, there exists a minimizing sequence $\{ \lambda_n \} = \{\lambda_n(u_n)\} \subset X^d$. Due to the Tikhonov term in the cost functional, we get that $\{ \lambda_n \}$ is bounded in $X^d$. Let $u_n$ be a minimiser of $\mathcal J^\gamma_{relax}$ for a corresponding $\lambda_n$. Such a minimiser exists because of Theorem \ref{Jrelaxexist}. Hence,
\begin{align*}
\mathcal J^\gamma_{relax}(u_n) & \leq \mathcal J^\gamma_{relax}(0)\\
\int_\Omega |\nabla u_n|_\gamma~ dx + C\int_\Omega |D_s u_n| + \sum_{i=1}^d \int_\Omega \left(\lambda_i\right)_n \phi_i(u_n,f)~ dx & \leq \sum_{i=1}^d \int_\Omega \left(\lambda_i\right)_n \phi_i(0,f)~ dx
\end{align*}
As before, from the coercivity condition on $f$ and the uniform bound on $\lambda_n$, we deduce that
\begin{align*}
C |D u_n|(\Omega) \leq C |D u_n|(\Omega) + \sum_{i=1}^d \int_\Omega \left(\lambda_i\right)_n \phi_i(u_n,f)~ dx &\leq \sum_{i=1}^d \int_\Omega \left(\lambda_i\right)_n \phi_i(0,f)~ dx\\
& \leq \frac{1}{2} \left(\sum_{i=1}^d \|(\lambda_i)_n\|_X +  \|\phi_i(0,f)\|_{L^2}^2\right)\\
& \leq C.
\end{align*}
Moreover, from the coercivity of $\phi_i$ in $u_n$ we get with a similar calculation that $u_n$ is uniformly bounded in $L^p$ for $p=1$ or $2$, and hence in particular in $L^1$. In sum, $u_n$ is uniformly bounded in $BV(\Omega)$ and hence, converges weakly $*$ in $BV(\Omega)$ and strongly in $L^1(\Omega)$.   The latter also gives pointwise convergence of $u_n$ and consequently $\phi_i(u_n,f)$ a.e. and hence we have
\begin{multline*}
\mathcal J^\gamma_{relax}(\hat{u},\hat{\lambda}) \\
:=  \int_\Omega |\nabla \hat{u}|_\gamma~ dx + C\int_\Omega |D_s \hat{u}| + \sum_{i=1}^d \int_\Omega \hat{\lambda}_i \phi_i(\hat{u},f)~ dx \\
\leq \liminf_n \left(\int_\Omega |\nabla u_n|_\gamma~ dx + C\int_\Omega |D_s u_n| + \sum_{i=1}^d \int_\Omega \left(\lambda_i\right)_n \phi_i(u_n,f)~ dx \right)\\
=  \liminf_n \mathcal J^\gamma_{relax}(u_n,\lambda_n). 
\end{multline*}
Since the cost functional is w.l.s.c., it follows, together with the fact that $\{ \lambda: \lambda \geq 0 \}$ is weakly closed, that $(\hat{\lambda},\hat{u})\in (X\cap\{\lambda_i\geq 0\}\times (BV(\Omega)\cap\mathcal A)$ is optimal for \eqref{eq:optimization problem1}.

\end{proof}

\begin{theorem}\label{thmgammaconv}
The sequence of functionals $\mathcal J^\gamma_{relax}$ in \eqref{Jrelaxedfunc} converges in the $\Gamma$- sense to the functional
$$
\mathcal J_{relax}(u) = 
\begin{cases}
\int_\Omega |\nabla u| + \sum_{i=1}^d \int_\Omega \lambda_i \phi_i(u)~ dx & u\in W^{1,1}(\Omega)\cap \mathcal A\\
+\infty & u\in BV(\Omega) \setminus (W^{1,1}(\Omega)\cap \mathcal A)
\end{cases}
$$
as $\gamma\rightarrow\infty$. Therefore, the unique minimiser of $\mathcal J^\gamma_{relax}$ converges to the unique minimiser of $\mathcal J_{relax}$ as $\gamma$ goes to infinity.
\end{theorem}

\begin{proof}
The proof is a standard result that follows from the fact that a decreasing point wise converging sequence of functionals $\Gamma$- converges to the lower semicontinuous envelope of the point wise limit \cite[Propostion 5.7]{DalMa93}. In fact, $\int_\Omega |\nabla u|_\gamma + \frac{1}{2\gamma}$ decreases in $\gamma$ and converges pointwise to $\int_\Omega |\nabla u|$. Then, for $u\in BV(\Omega)$ the functional $\mathcal J^\gamma_{relax}$ (being the lower-semicontinuous envelope of $\mathcal J^\gamma_{ext}$) $\Gamma$- converges to the functional $\mathcal J_{relax}$ in Theorem \ref{thmgammaconv}. The latter is the lower-semicontinous envelope of the functional in \eqref{general denoise}.
\end{proof}

%

Although Theorem \ref{thmgammaconv} provides a convergence result for the regularized TV subproblems, it is not sufficient to conclude convergence of the optimal regularized weights. For this we need the continuity of the solution map $\lambda \rightarrow u(\lambda)$. Up to our knowledge, no sufficient continuity results for the control-to-state map in the case of a total variation minimiser as the state are known. There are various contributions in this directions \cite{CE05,Morozov,VO96,YP03} which are -- as they stand -- not strong enough to prove the desired result in our case. Indeed, this is a matter of future research. 

\section{Optimization problem in $H_0^1(\Omega)$}\label{sec:h1analysis}
In order to obtain continuity of the solution map and, hence, convergence of the regularized optimal parameters, we proceed in an alternative way and move, from now on, to a Hilbert space setting. Specifically, we replace the minimisation problem \eqref{general denoise} by the following elliptic-regularized version of it:
\begin{equation}\label{elliptic denoise}
\min_u {\left(\frac{\varepsilon}{2} \|Du\|^2_{L^2} + |Du|(\Omega) + \sum_{i=1}^d \int_\Omega
\lambda_i~ \phi_i (u)\; dx \right)}.
\end{equation}
where $0 < \varepsilon \ll 1$ is an artificial diffusion parameter.

A necessary and sufficient optimality condition for \eqref{elliptic denoise} is
given by the following elliptic variational inequality:
\begin{multline}\label{eq:original VI}
\varepsilon (Du,D(v-u))_{L^2} + \sum_{i=1}^d \int_{\Omega} \lambda_i \phi_i'(u)(v-u) ~dx\\+
\int_{\Omega}|Dv| ~dx-\int_{\Omega}|Du| ~dx \geq 0 \,\, \text{ for all }v \in
H_0^1(\Omega).
\end{multline}
Note that by adding the coercive term, we implicitely impose the solution space $H_0^1(\Omega)$ (see \cite{HintermuellerKunisch2004}).

Our aim is to determine the optimal choice of parameters $\lambda_i,
~i=1,...,d,$ by solving the following optimization problem:
\begin{subequations} \label{eq:optimization problem}
\begin{equation}
\min_{\lambda_i \geq 0, ~{i=1,...,d}} ~ g(u) + \beta \sum_{i=1}^d \|\lambda_i\|_{X}^2
\end{equation}
subject to
\begin{multline}  \label{eq:optimization problem eq 2}
\varepsilon (Du,D(v-u))_{L^2} + \sum_{i=1}^d \int_{\Omega} \lambda_i~
\phi_i'(u) (v-u) ~dx\\+ \int_{\Omega} |Dv| ~dx-\int_{\Omega} 
|Du| ~dx \geq 0 \, \, \text{ for all } v \in H_0^1(\Omega),
\end{multline}
\end{subequations}
where the space $X$ corresponds to $\mathbb R$ in the case of scalar parameters or to a Hilbert function space in the case of distributed functions. Problem \ref{eq:optimization problem} corresponds to an optimization problem governed by a variational inequality of the second kind (see \cite{Delosreyes2009} and the references therein).


Next, we perform the analysis of the optimization problem
\eqref{eq:optimization problem}. After proving existence of an optimal
solution, a regularization approach will be also proposed in this context. We will prove
the continuity of the control-to-state map and, based on it,
convergence of the regularized images and the optimal regularized
parameters. In the case of a smoother regularization of the TV term,
also differentiablity of the solution operator will be verified, which will lead us afterwards to a first order optimality system characterizing the optimal solution to \eqref{eq:optimization problem}.

We start with the following existence theorem.
\begin{theorem} \label{thm:existence of optima}
There exists an optimal solution for problem \eqref{eq:optimization problem}.
\end{theorem}
\begin{proof}
Let $\{ \lambda_n \} \subset X^d$ be a minimizing sequence. Due to the Tikhonov term in the cost functional, we get that $\{ \lambda_n \}$ is bounded in $X^d$. From \eqref{eq:optimization problem eq 2} we additionally get that the sequence of images $\{ u_n \}$ satisfy
\begin{multline}
\varepsilon \|u_n\|_{H_0^1}^2+ \sum_{i=1}^d \int_{\Omega} {\lambda_n}_i~
\left[ \phi_i'(u_n)-\phi_i'(0) \right] u_n ~dx\\ + \int_{\Omega} |Du_n| ~dx \leq - \sum_{i=1}^d \int_{\Omega} {\lambda_n}_i~ \phi_i'(0) u_n ~dx,
\end{multline}
which, due to the monotonicity of the operators on the left hand side, implies that
\begin{equation}
\varepsilon \|u_n\|_{H_0^1}^2 \leq \sum_{i=1}^d \| {\lambda_n}_i\|_X \|\phi_i'(0)\|_{L^r} \|u_n \|_{L^p},
\end{equation}
for $2 < p < +\infty$ and $r=\frac{2p}{p-2}$. Thanks to the embedding $H_0^1(\Omega) \hookrightarrow L^p(\Omega)$, for all $1 \leq p < +\infty$, we get that 
\begin{equation}
\|u_n\|_{H_0^1} \leq C \sum_{i=1}^d \| {\lambda_n}_i\|_X \|\phi_i'(0)\|_{L^r},
\end{equation}
for some constant $C>0$. Consequently,
\begin{equation}
\{u_n\} \text{ is uniformly bounded in }H_0^1(\Omega).
\end{equation}

Therefore, there exists a subsequence $\{ (u_n, \lambda_n) \}$ which converges weakly in $H_0^1(\Omega) \times X^d$ to a limit point $(\hat u, \hat \lambda)$. Moreover, $u_n \to \hat u$  strongly in $L^p(\Omega)$ and, thanks to the continuity of $\phi'$, also 
\begin{equation}
\phi'(u_n)u_n \weak \phi'(\hat u) \hat u \hspace{0.5cm} \text{ strongly in }L^{\frac{p}{2}}(\Omega)
\end{equation}

Consequently, thanks to the continuity of $\phi'$ and the properties of $a(\cdot, \cdot)$, we get
\begin{align*}
a(\hat u, \hat u) + \sum_{i=1}^d \int_\Omega \hat \lambda_i \phi_i'(\hat u) \hat u &+ \int_\Omega |D \hat u|\\ &\leq  \liminf a(u_n, u_n) + \sum_{i=1}^d \int_{\Omega} {\lambda_n}_i \phi_i'(u_n) u_n + \int_\Omega |D u_n|\\
& \leq  \liminf a(u_n, v) + \sum_{i=1}^d \int_{\Omega} {\lambda_n}_i \phi_i'(u_n) v + \int_\Omega |D v|\\
& = a(\hat u, v)+ \sum_{i=1}^d \int_\Omega \hat \lambda_i \phi_i'(\hat u) v + \int_\Omega |D v|.
\end{align*}
Since the cost functional is w.l.s.c., it follows, together with the fact that $\{ \lambda: \lambda \geq 0 \}$ is weakly closed, that $(\hat \lambda, \hat u)$ is optimal for \eqref{eq:optimization problem}.
\end{proof}

Next, we consider the following family of regularized problems:
\begin{equation}\label{eq:reg. VI}
\epsilon (D u_\gamma, Dv)+ (h_\gamma(D u_\gamma), D v) + \sum_{i=1}^d \int_\Omega \lambda_i ~\phi_i'(u_\gamma)v =0, \forall v \in H_0^1(\Omega),
\end{equation}
where $h_\gamma(D u_\gamma)$ corresponds to an active-inactive-set approximation of the subdifferential of $|D u_\gamma|$, i.e., $h_\gamma$ coincides with an element of the subdifferential up to a small neighborhood of 0. The most natural choice is the function 
\begin{equation}\label{eq:derhub}
h_\gamma (z)= \frac{\gamma z}{\max(1,\gamma |z|)},
\end{equation} 
which corresponds to the derivative of the Huber function defined in \eqref{eq:huber}.
An alternative regularization is given by the $C^1$ function
\begin{equation}\label{eq:local reg. of q}
h_{\gamma}(z)=
\begin{cases}
g\frac{z}{|z|} &\text{ if }~\gamma |z| \geq g+ \frac{1}{2\gamma}\\
 \frac{z}{|z|} (g- \frac{\gamma}{2} (g- \gamma |z|+\frac{1}{2\gamma})^2) &\text{ if }~g-\frac{1}{2\gamma}\leq \gamma |z| \leq g+\frac{1}{2\gamma}\\
\gamma z &\text{ if }~\gamma |z| \leq g-\frac{1}{2\gamma},
\end{cases}
\end{equation}
where $g$ is a positive parameter. The latter smoothing for the total variation is going to be used in Proposition \ref{prop:linearisation} where differentiability of the regulariser is needed.

\begin{remark}\label{rem:properties VI}
For a fixed $\lambda$, it can be verified that \eqref{eq:reg. VI} has a unique solution. Moreover, the sequence of regularized solutions $\{ u_\gamma \}$ converges strongly in $H_0^1(\Omega)$ to the solution of \eqref{eq:original VI} (cf. \cite{Delosreyes2009}).
\end{remark}

Based on the regularized problems \eqref{eq:reg. VI}, we now focus on the following optimization problem:
\begin{subequations} \label{eq:reg. optimization problem}
\begin{equation}
\min_{\lambda_i \geq 0, ~{i=1,...,d}} ~ g(u) + \beta \sum_{i=1}^d \|\lambda_i\|_{X}^2
\end{equation}
subject to
\begin{equation}
\epsilon (D u_\gamma, Dv)+ (h_\gamma(D u_\gamma), D v) + \sum_{i=1}^d \int_\Omega \lambda_i ~\phi_i'(u_\gamma)v =0, \forall v \in H_0^1(\Omega),\label{eq:regeps. VI}
\end{equation}
\end{subequations}
where $0<\epsilon\ll 1$. In this setting we can prove the following continuity and convergence results.

\begin{proposition} \label{lem:continuity of reg mapping}
Let $\{ \lambda_n \}$ be a sequence in $X^d$ such that $\lambda_n \weak \hat \lambda$ weakly in $X^d$ as $n \to \infty$. Further, let $u_n:=u_\gamma (\lambda_n)$ denote the solution to \eqref{eq:regeps. VI} associated with $\lambda_n$ and $\hat u:= u_\gamma(\hat \lambda)$. Then 
$$u_n \to \hat u \text{ strongly in }H_0^1(\Omega).$$
\end{proposition}
\begin{proof}
Since $\lambda_n \weak \hat \lambda$ weakly in $X$, it follows, by the principle of uniform boundedness, that $\{ \lambda_n \}$ is bounded in $X$. From \eqref{eq:regeps. VI} we additionally get that
\begin{multline}
\varepsilon \|u_n\|_{H_0^1}^2+ \sum_{i=1}^d \int_{\Omega} {\lambda_n}_i~
\left[ \phi_i'(u_n)-\phi_i'(0) \right] u_n ~dx\\ + (h_\gamma(D u_n), D u_n) \leq - \sum_{i=1}^d \int_{\Omega} {\lambda_n}_i~ \phi_i'(0) u_n ~dx,
\end{multline}
which, proceeding as in the proof of Theorem \ref{thm:existence of optima}, implies that
\begin{equation}
\|u_n\|_{H_0^1} \leq C \sum_{i=1}^d \| {\lambda_n}_i\|_X \|\phi_i'(0)\|_{L^r},
\end{equation}
for some constant $C>0$. Hence, $\{u_n\} \text{ is uniformly bounded in }H_0^1(\Omega).$

Consequently, there exists a subsequence (denoted the same) and a limit $\hat u$ such that
$$u_n \weak \hat u \text{ in }H_0^1(\Omega) \hspace{0.5cm} \text{ and } \hspace{0.5cm} u_n \to \hat u \text{ in }L^p(\Omega), ~1 \leq p < +\infty.$$

Thanks to the structure of the regularized VI \eqref{eq:regeps. VI} it also follows (as in the proof of Theorem \ref{thm:existence of optima}) that $\hat u$ is solution of the regularized VI associated with $\hat \lambda$. Since the solution to \eqref{eq:regeps. VI} is unique, it additionally follows that the whole sequence $\{ u_n \}$ converges weakly towards $\hat u$.

To verify strong convergence, we take the difference of the variational equations satisfied by $u_n$ and $\hat u$ and obtain that
\begin{multline*}
\varepsilon (D u_n -D \hat u, Dv)+ (h_\gamma(D u_n)-h_\gamma(D \hat u), D v) \\= \int_\Omega \left[\hat \lambda ~\phi'(\hat u) -\lambda_n ~\phi'(u_n) \right] v ~dx, \forall v \in H_0^1(\Omega).
\end{multline*}
Adding the term $-\lambda_n ~\phi'(\hat u)$ on both sides of the latter yields
\begin{multline*}
\varepsilon (D u_n -D \hat u, Dv)+ (h_\gamma(D u_n)-h_\gamma(D \hat u), D v) \\ \int_\Omega \left[\lambda_n ~\phi'(u_n) -\lambda_n ~\phi'(\hat u) \right]v ~dx= \int_\Omega \left[\hat \lambda ~\phi'(\hat u) -\lambda_n ~\phi'(\hat u) \right] v ~dx, \forall v \in H_0^1(\Omega).
\end{multline*}

Choosing $v= u_n- \hat u$ and thanks to the monotonicity of the operator on the left hand side, we then obtain that
\begin{equation*}
\varepsilon \|D u_n -D \hat u \|^2 \leq \left| \int_\Omega \left[\hat \lambda ~\phi'(\hat u) -\lambda_n ~\phi'(\hat u) \right] (u_n- \hat u) ~dx \right|,
\end{equation*}
which thanks to the strong convergence $u_n \to \hat u \text{ in }L^p(\Omega), ~1 \leq p < +\infty,$ and the regularity $\phi'(\hat u) \in L^{\frac{p}{2}}(\Omega)$, implies the result.
\end{proof}

\begin{theorem}
There exists an optimal solution for each regularized problem \eqref{eq:reg. optimization problem}. Moreover, the sequence $\{ \lambda_\gamma \}$ of regularized optimal parameters is bounded in $X^d$ and every weakly convergent subsequence converges towards an optimal solution of \eqref{eq:optimization problem}.
\end{theorem}
\begin{proof}
Let $\{ \lambda_n\}$ be a minimizing sequence. From the structure of the cost functional and the properties of \eqref{eq:regeps. VI} it follows that the sequence is bounded. Consequently, there exists a subsequence (denoted the same) and a limit $\lambda^*$ such that $\lambda_n \weak \lambda^*$ weakly in $X^d$. From Proposition \ref{lem:continuity of reg mapping} and the weakly lower semicontinuity of the cost functional, optimality of $\lambda^*$ follows, and, therefore, existence of an optimal solution.

Let now $\{ \lambda_\gamma \}_{\gamma >0}$ be a sequence of optimal solutions to \eqref{eq:reg. optimization problem}. Since $(0,0) \in H_0^1(\Omega) \times X^d$ is feasible for each $\gamma >0$, it follows that
$$J(u_\gamma(\lambda_\gamma),\lambda_\gamma) \leq J(u_\gamma(0),0) =J(0,0).$$
Thanks to the Tikhonov term in the cost functional it then follows that $\{ \lambda_\gamma \}_{\gamma >0}$ is bounded.

Let $\hat \lambda$ be the limit point of a weakly convergent subsequence (also denoted by $\{ \lambda_\gamma \}$). From Remark \ref{rem:properties VI} and Proposition \ref{lem:continuity of reg mapping} it follows, by using the triangle inequality, that
$$u_{\gamma}(\lambda_\gamma) \to \hat u \hspace{0.3cm} \text{ strongly in }  H_0^1(\Omega) \hspace{0.3cm} \text{ as }\gamma \to \infty,$$
where $\hat u$ denotes the solution to \eqref{eq:original VI} associated with $\hat \lambda$.

From the weakly lower semicontinuity of the cost functional, we finally get that
\begin{align*}
J(\hat u, \hat \lambda)  \leq \liminf_{\gamma \to \infty} J(u_\gamma(\lambda_\gamma), \lambda_\gamma) \leq \liminf_{\gamma \to \infty} J(u_\gamma(\bar \lambda), \bar \lambda) =J(\bar u, \bar \lambda),
\end{align*}
where $\bar \lambda$ is an optimal solution to \eqref{eq:reg. optimization problem}.
\end{proof}

The next proposition is concerned with the differentiability of the solution operator. This result will lead us thereafter (see Theorem \ref{thm:optimality system}) to get an expression for the gradient of the cost functional and also to obtain an optimality system for the characterization of the optimal solutions to \eqref{eq:optimization problem}.
\begin{proposition}\label{prop:linearisation}
Let $G_\gamma: X^d \mapsto H_0^1(\Omega)$ be the solution operator, which assigns to each parameter $\lambda$ the corresponding solution to the regularized VI \eqref{eq:reg. VI}, with the function $h_\gamma$ given by \eqref{eq:local reg. of q}. Then the operator $G_\gamma$ is G\^ateaux differentiable and its derivative at $\bar \lambda$, in direction $\xi$, is given by the unique solution $z \in H_0^1(\Omega)$ of the following linearized equation:
\begin{multline} \label{eq:linearized equation}
 \epsilon (D z, Dv)+ (h_\gamma'(D \bar u) D z, D v) + \sum_{i=1}^d \int_\Omega \lambda_i ~\phi_i''(\bar u) ~z ~v ~dx\\ + \sum_{i=1}^d \int_\Omega \xi_i ~\phi_i'(\bar u) ~v ~dx=0, ~\text{ for all } v \in H_0^1(\Omega).
\end{multline}
\end{proposition}
\begin{proof}
Existence and uniqueness of a solution to \eqref{eq:linearized equation} follows from Lax-Milgram theorem by making use of the monotonicity properties of $h_\gamma$ and $\phi_i'$.

Let $\xi \in X^d$, and let $y_t$ and $y$ be the unique solutions to \eqref{eq:reg. VI}
correspondent to $\lambda +t \xi$ and $\lambda$, respectively. By taking the difference between both equations, it follows that
\begin{multline}
\epsilon (D(u_t-u),D(u_t-u))+ \left( h_\gamma (\E u_t) -h_\gamma (\E u), \E (u_t-u)
\right)\\ +\lambda_t (\phi'(u_t)-\phi'(u),u_t-u) = -t (\xi \phi'(u), u_t-u),
\end{multline}
which, by the monotonicity of $h_\gamma$ and $\phi'$ yields that
\begin{equation}
\k\| u_t-u \|^2_{H_0^1} \leq t~\| \xi \|_{X^d} \|\phi'(u) \|_{L^r} \| u_t-u
\|_{L^p}.
\end{equation}
Therefore, the sequence $\{ z_t \}_{t>0}$, with $z_t:= \frac{y_t-y}{t}$, is
bounded and there exists a subsequence (denoted the same) such that $z_t
\weak z$ weakly in $H_0^1(\Omega)$.

Using the mean value theorem in integral form we get that
\begin{align}
a(z_t,w)&+\frac{1}{t} \left( h_\gamma (\E u_t) -h_\gamma (\E u), \E w
\right)+ \frac{1}{t} \left( \lambda_t (\phi'(u_t)- \phi'(u)),w \right)\\
&= a(z_t,w)+ \int \limits_\Omega \escu{h_\gamma' ( \vartheta_t) \E z_t}{\E w}~dx+ \int \limits_\Omega \lambda_t \phi''(\zeta_t) w ~dx \label{eq:derivative of reg. function in adjoint eq.}
\\&= -(\xi \phi'(u), w), \text{ for all
}w \in V,\label{eq:derivative of reg. function in adjoint eq. 2}
\end{align}
where $\vartheta_t(x)= \E u(x)+ \rho_t(x)(\E u_t(x) -\E u(x))$, with $0 \leq
\rho_t(x) \leq 1$, and $\zeta_t(x)= u(x)+ \varrho_t(x)(u_t(x) -u(x))$, with $0 \leq
\varrho_t(x) \leq 1$.

From the continuity of the bilinear form it follows that $a(z_t,w) \to
a(z,w)$, for all $w \in V$. Additionally, from the consistency of the regularization (see Remark \ref{rem:properties VI}), $u_t \to u$ strongly in $H_0^1(\Omega)$
and, therefore, $\vartheta_t \to \E u$ strongly in $\mathbb L^2(\Omega)$ and $\zeta_t \to u$ strongly in $H_0^1(\Omega)$.

Introducing the function
\begin{equation*}
\chi_\gamma(x) : =\begin{cases}
g &\text{ if } \gamma | x | \geq g +\frac{1}{2 \gamma},\\
g- \frac{\gamma}{2}(g-\gamma |x|+ \frac{1}{2 \gamma})^2 &\text{ if } \left| \gamma | x | -g \right| \leq \frac{1}{2 \gamma},\\
\gamma |x| &\text{ if } \gamma | x | \leq g -\frac{1}{2 \gamma},
\end{cases}
\end{equation*}
the regularizing function may be written as $h_{\gamma}(x)=\frac{x}{|x|} ~ \chi_\gamma(x)$ and 
the second term in \eqref{eq:derivative of reg. function in adjoint eq.} can be expressed as
\begin{multline}
\int \limits_\Omega \escu{h_{\gamma}' (\vartheta_t) \E z_t}{\E w}~dx= \int \limits_\Omega \chi_\gamma(\vartheta_t) \escu{\frac{\E w}{|\vartheta_t|}- \frac{\escu{\vartheta_t}{\E w}}{|\vartheta_t|^2}\frac{\vartheta_t}{|\vartheta_t|}}{\E z_t}~dx \nonumber
\\ \hspace{1cm}+ \int \limits_\Omega \chi_\gamma'(\vartheta_t)[\E w]
\escu{\frac{\vartheta_t}{|\vartheta_t|}}{\E z_t}~dx.
\end{multline}

Let $\Phi : \mathbb R^2 \mapsto \mathbb R^2$ be the operator defined by
\begin{equation} \label{eq:superposition operator Bingham 2}
\Phi(\xi):= \chi_\gamma(\xi) \frac{w}{|\xi|}- \frac{\escu{\xi}{w}}{|\xi|^2}\frac{\xi}{|\xi|} + \chi_\gamma'(\xi)[w]
\frac{\xi}{|\xi|},
\end{equation}
with $w \in \mathbb R^2$. When considered from $\mathbb L^q(\Omega)$ to $\mathbb
L^q(\Omega)$, $\Phi$ is a continuous superposition operator. Therefore, since $\vartheta_t \to \E u$ strongly in $\mathbb L^2(\Omega)$ and thanks to the weak
convergence of $z_t$ and continuity of $\phi''$, we may pass to the limit in \eqref{eq:derivative of reg. function in adjoint eq.}-\eqref{eq:derivative of reg. function in adjoint eq. 2} and obtain that
\begin{multline} \label{eq:proof limit satisfies lin eq}
 a(z,w)+ \int \limits_\Omega \escu{h_\gamma' ( \E u) \E z}{\E w}~dx \\= - \int \limits_\Omega \lambda \phi''(\zeta) w ~dx -\int \limits_\Omega \xi \phi'(u) ~w~dx, \text{ for all } w
\in H_0^1(\Omega).
\end{multline}
Consequently, $z \in  H_0^1(\Omega)$ corresponds to the unique solution of the linearized equation.

Using again the function $\chi_\gamma$, the operator on the left hand side of equation \eqref{eq:proof limit satisfies lin eq} may be written as $$\int_\Omega {\E w}^T M(x) {\E z} ~dx,$$ where
\begin{multline} \label{eq: matrix lin eq}
M(x):=2 \mu I + \frac{\chi_\gamma(\E y)}{|\E y|} I- \frac{\chi_\gamma(\E y)}{|\E y|^3} \left[ {\E y}(x) {\E y}(x)^T \right]\\+ \chi_{\mathcal S^\gamma}\frac{\gamma^2}{|\E y|^2}(g- \gamma |\E y| +\frac{1}{2 \gamma}) \left[ {\E y}(x) {\E y}(x)^T \right] +\chi_{\mathcal I^\gamma} \frac{\gamma}{|\E y|^2} \left[ {\E y}(x) {\E y}(x)^T \right],
\end{multline}
where $\chi_{\mathcal S^\gamma}$ and $\chi_{\mathcal I^\gamma}$ correspond to the indicator functions of the sets $\mathcal S^\gamma:=\{ x \in \Omega: \left| \gamma |\E y| - g \right| < \frac{1}{2 \gamma} \}$ and $\mathcal I^\gamma:=\{ x \in \Omega: \gamma |\E y| \leq g -\frac{1}{2 \gamma} \}$, respectively.

Similarly, by replacing ${\E y}$ with ${\vartheta_t}$ in \eqref{eq: matrix lin eq} a matrix denoted by $M_t$ is obtained. Both matrices $M$ and $M_t$ are symmetric and positive definite. Using
Cholesky decomposition we obtain lower triangular matrices $L_t$ and $L$ such that
$$M_t(x)=L_t(x) ~L_t^T(x) \hspace{0.3cm} \text{ and } \hspace{0.3cm} M(x)=L(x) ~L^T(x).$$
Proceeding as in \cite[pp.~30-31]{CasasFer1993} (see also \cite[Thm.~6.1]{Delosreyes2009}), strong convergence of $z_t \to z$ in $H_0^1(\Omega)$ is obtained, and, thus, also G\^ateaux differentiability of $G_\gamma$.
\end{proof}

\begin{theorem}[Optimality system] \label{thm:optimality system}
Let $(\bar \lambda, \bar u)$ be an optimal solution to problem \eqref{eq:reg. optimization problem} with $X^d=\mathbb R^d$. There exist Lagrange multipliers $(p,\mu) \in H_0^1(\Omega) \times \mathbb R^d$ such that the following optimality system holds:
\begin{subequations}
 \begin{equation} \label{eq:optimality condition state equation}
  \epsilon (D \bar u, Dv)+ (h_\gamma(D \bar u), D v) + \sum_{i=1}^d \int_\Omega \bar \lambda_i ~\phi_i'(\bar u)v ~dx =0, \forall v \in H_0^1(\Omega),
 \end{equation}
\begin{multline} \label{eq:optimality condition adjoint equation}
 \epsilon (D p, Dv)+ (h_\gamma'(D \bar u)^* D p, D v)\\ + \sum_{i=1}^d \int_\Omega \lambda_i ~\phi_i''(\bar u) ~p ~v ~dx =-(g'(\bar u),v), \forall v \in H_0^1(\Omega),
\end{multline}
\begin{equation}
 \mu_i= 2 \beta \bar \lambda_i + \int_\Omega p \phi_i'(\bar u)~dx, \hspace{1cm} i=1,...,d,
\end{equation}
\begin{equation} \label{eq:optimality condition complementarity system}
 \mu_i \geq 0, ~ \lambda_i \geq 0, ~\mu_i \lambda_i =0, \hspace{1cm} i=1,...,d.
\end{equation}
\end{subequations}
\end{theorem}
\begin{proof}
Consider the reduced cost functional
\begin{equation}
f(\lambda)=g(G_\gamma(\lambda))+ \beta \| \lambda \|_{\mathbb{R}^d}^2.
\end{equation}
Thanks to the optimality of $\bar \lambda$ and the differentiability of both $G_\gamma$ and $g$ it follows that
\begin{equation} \label{eq:optimality condition variational ineq}
\nabla f(\bar \lambda)^T(\xi- \bar \lambda) \geq 0, \text{ for all } \xi \geq 0.
\end{equation}
Let $p \in H_0^1(\Omega)$ be the unique solution to the adjoint equation:
\begin{multline} \label{eq:adjoint equation}
 \epsilon (D p, Dv)+ (h_\gamma'(D \bar u)^* D p, D v)\\ + \sum_{i=1}^d \int_\Omega \lambda_i ~\phi_i''(\bar u) ~p ~v ~dx =-(g'(\bar u),v), \forall v \in H_0^1(\Omega).
\end{multline}
Indeed, existence and uniqueness of a solution to \eqref{eq:adjoint equation} follows from the Lax-Milgram theorem, similarly as for the linearized equation.

Using the adjoint equation it follows that
\begin{align*}
\nabla f(\bar \lambda)^T \xi &= (g'(\bar u),G_\gamma'(\bar u)\xi)+ 2 \beta \bar \lambda^T \xi\\
&= 2 \beta \bar \lambda^T \xi - \epsilon (D p, Dv)- (h_\gamma'(D \bar u)^* D p, D v) - \sum_{i=1}^d \int_\Omega \lambda_i ~\phi_i''(\bar u) ~p ~v ~dx,
\end{align*}
which, utilizing the linearized equation \eqref{eq:linearized equation}, yields that
\begin{equation} \label{eq: characterization of gradient}
\nabla f(\bar \lambda)^T \xi =2 \beta \bar \lambda^T \xi + \sum_{i=1}^d \int_\Omega \phi_i'(\bar u) p ~\xi_i.
\end{equation}
Let $\mu_i:= 2 \beta \bar \lambda_i + \int_\Omega \phi_i'(\bar u) p,
~i=1,...,d.$ From \eqref{eq:optimality condition variational ineq} and \eqref{eq: characterization of gradient} it then follows that
$$\mu^T(\xi- \bar \lambda) \geq 0, \text{ for all } \xi \geq 0,$$
which is equivalent to the complementarity system \eqref{eq:optimality condition complementarity system}.
\end{proof}

\begin{remark}
In the case of a general parameter Hilbert space $X^d$, an optimality system
constituted by equations \eqref{eq:optimality condition state equation},
\eqref{eq:optimality condition adjoint equation} and the variational
inequality $$
2 \beta( \bar \lambda, \xi- \bar \lambda)_{X^d}  + ( \phi'(\bar u) p,
\xi- \bar \lambda)_{\mathbb L^2} \geq 0, \text{ for all } \xi \in X^d
: \xi \geq 0 \text{ a.e.}$$ is
obtained.
\end{remark}

\section{Numerical solution of the optimization problem}
In this section we focus on the numerical solution of the optimization
problem \eqref{eq:optimization problem}. For the determination of the
optimal parameter values we consider a projected
BFGS (Broyden-Fletcher-Goldfarb-Shanno) method. In
each computational experiment, the state equation is solved by means of a Newton type
algorithm (specified in each case), while a forward finite differences
quotient is used for the evaluation of the gradient of the cost
functional. Due to the high computational cost of solving the state
equation, a fixed line search parameter value $\alpha=0.5$ was considered in
combination with the BFGS method. The linear systems in each Newton iteration are solved
exactly by using a LU decomposition for band matrices. 

\subsection{Gaussian noise}
As a first example we consider the determination of a single
regularization parameter, i.e., $\lambda \in \mathbb R$. The optimization problem takes the following form:
\begin{subequations} \label{eq:gaussian noise problem}
\begin{equation}
\min  ~\frac{1}{2} \|u - u_{o}\|^2_{L^2} + \beta \lambda^2 
\end{equation}
subject to:
\begin{multline}
\varepsilon (Du,D(v-u))_{L^2} +\int_{\Omega} \lambda (u-u_{n}) (v-u) ~dx\\+ \int_{\Omega} |Dv| ~dx-\int_{\Omega}  |Du| ~dx \geq 0, \forall v \in H_0^1(\Omega),
\end{multline}
\end{subequations}
where $u_o$ and $u_n$ denote the original and noisy images respectively. The problem consists therefore in the optimal choice of the TV regularization parameter, if the original image is known in advance. This is a toy example for proof of concept only. In practice this image would be replaced by a training set of images as motivated in the Introduction. 

For the numerical solution of the regularized variational inequality we utilize the
primal-dual algorithm developed in \cite{HintermuellerStadler}, which
was proved to be globally as well as locally superlinear convergent.

The results for the parameter values $\beta= 1 \times 10^{-10}, \varepsilon = 1 \times 10^{-12}, \gamma=100$ and $h=1/177$ are shown in Figure \ref{fig:experiment1} together with the noisy image distorted by Gaussian noise with zero mean and variance 0.002. 
\begin{figure} \label{fig:experiment1}
\begin{center}
\includegraphics[width=6.37cm]{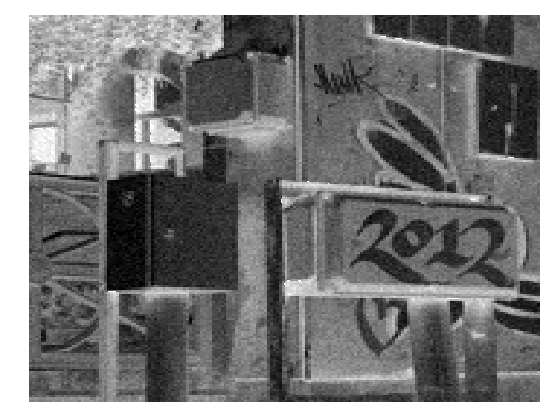}
\includegraphics[width=6.37cm]{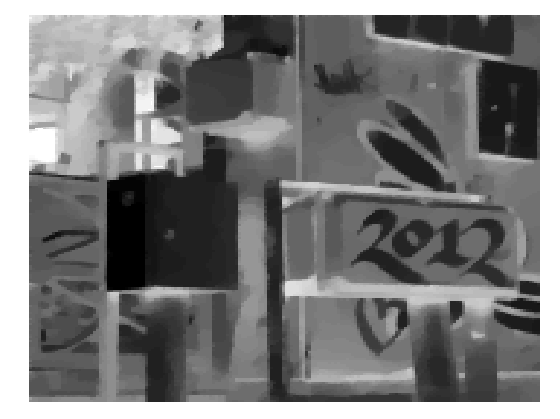}\\
\hspace{0.01cm} \includegraphics[height=4.65cm,width=6.15cm]{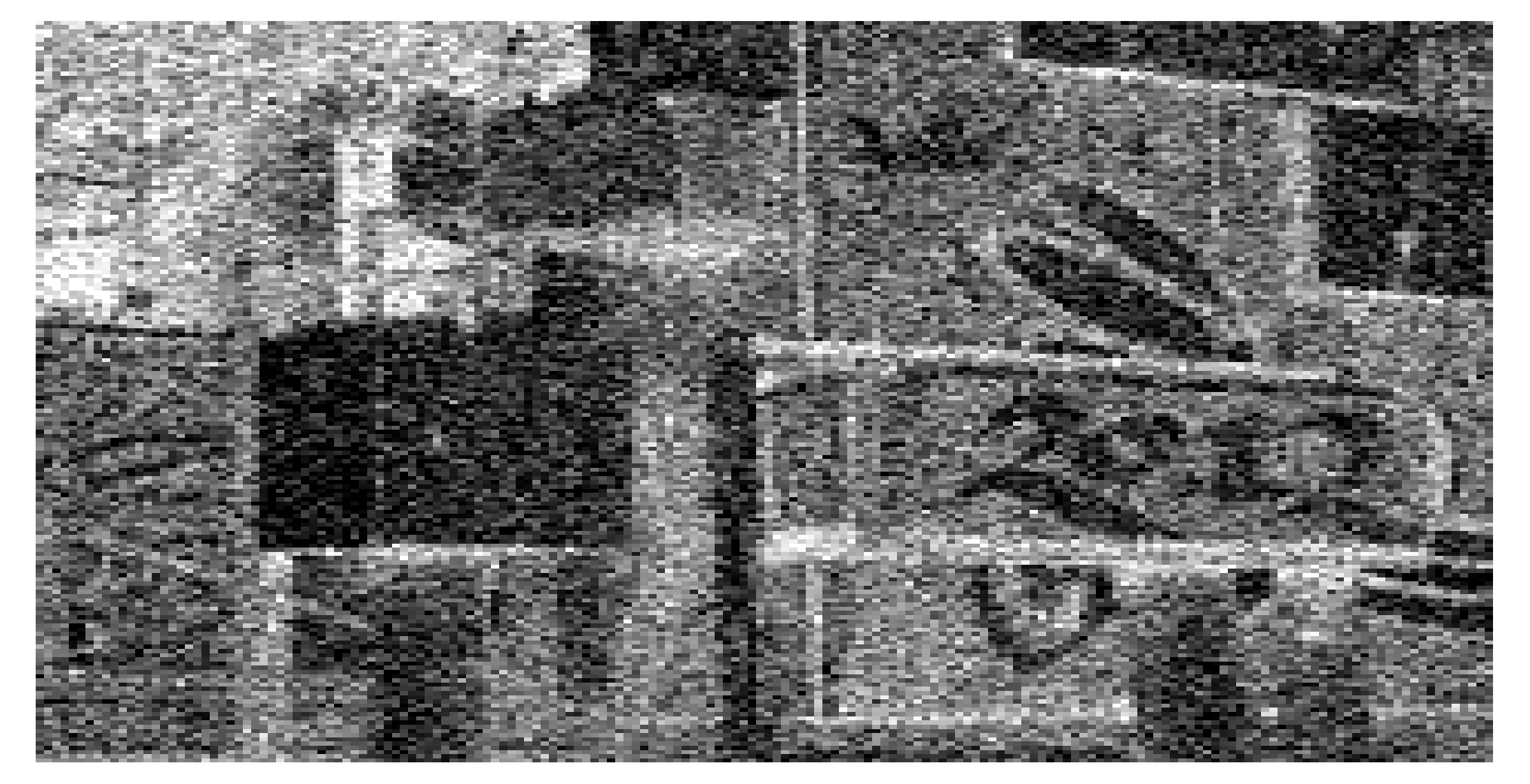} \hspace{0.15cm} \includegraphics[height=4.65cm,width=6.15cm]{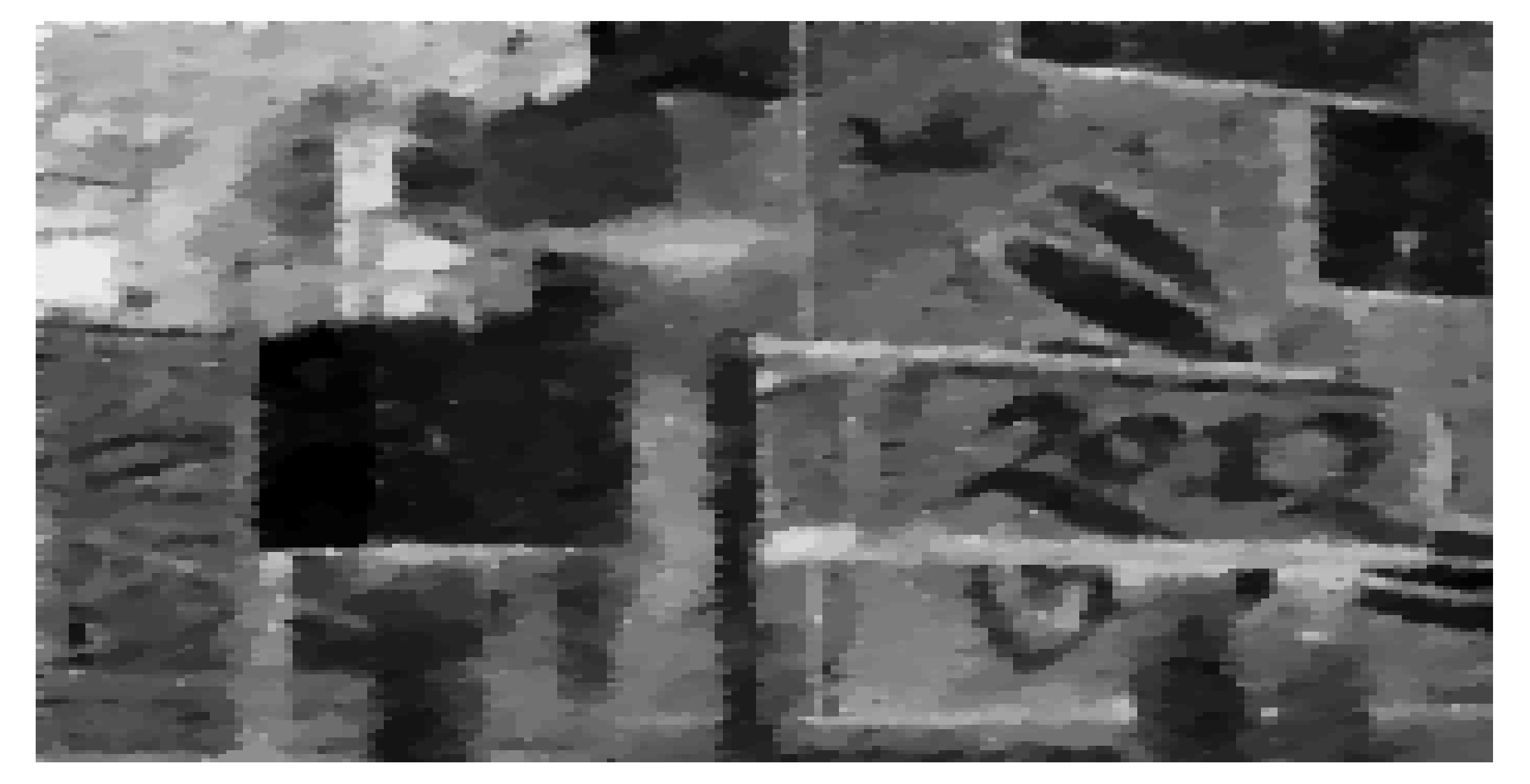}
\end{center}
\caption{Noisy (left) and denoised (right) images. Noise variance:
  0.002 (first row) and 0.02 (last row).}
\end{figure}
The computed optimal parameter for this problem is $\lambda^*=2980$
For a noise of mean 0 and variance 0.02, and the same regularization
parameters as in the previous experiment, the optimal image is
obtained with the weight $\lambda^* =1770.9$. The noisy and denoised
images are given in Figure \ref{fig:experiment1}.

By increasing the variance in the noise, the optimal values of the
weight differ significantly. This is intuitively clear, since as the
image becomes noisier there is less original information that can be directly
obtained. When that happens, the TV regularization plays an
increasingly important role.

The question of robustness of the optimal parameter value deserves
also to be tested. In Table \ref{table: mesh dependence} we compute
the optimal weight for different sources of the noisy image (different
total number of pixels). The value
of the optimal weight increases together with the size of the image
from which the information is obtained. The variation remains however
small, implying a robust behavior of the values. 
\begin{table}
\centering
\begin{tabular}{|c|c|c|c|c|c|c|} \hline
$\hspace{.2cm} \# \text{pixels} \hspace{.2cm}$ & \hspace{.2cm} 60 \hspace{.2cm} &
\hspace{.2cm} 65 \hspace{.2cm} & \hspace{.2cm} 70 \hspace{.2cm} & \hspace{.2cm}
75 \hspace{.2cm} &\hspace{.2cm} 80 \hspace{.2cm} &\hspace{.2cm} 85
\hspace{.2cm}\\
\hline $\lambda^* $ &674.6 &742.9 & 788.9 &855.0 &885.8 &933\\
\hline
\end{tabular}
\vspace{0.2cm}\caption{Optimal weight vs. mesh size; $\mu=1e-15$, $\gamma=100$, $\beta=1e-10$.} \label{table: mesh dependence}
\end{table}

\subsection{Magnetic resonance imaging}
Gaussian noise images typically arise within the framework of magnetic
resonance imaging (MRI). The challenge in this case consists in
training the machines so that a clearer image is obtained. The magnetic
resonance images seem to be the natural choice for our methodology, since a training set of
images is often at hand.

For such a training set we consider the solution of problem \eqref{eq:gaussian noise problem}. In Figure \ref{fig:experimentmri} the noisy images together with the final optimized ones for a brain scan are shown. For this experiments a mesh step size of $h=1/250$ was considered. The Tikhonov regularization parameter took the value $\beta=10^{-10}$, while the Huber regularization parameter was chosen as $\gamma=100$.
\begin{figure} \label{fig:experimentmri}
\begin{center}
\includegraphics[width=6.2cm]{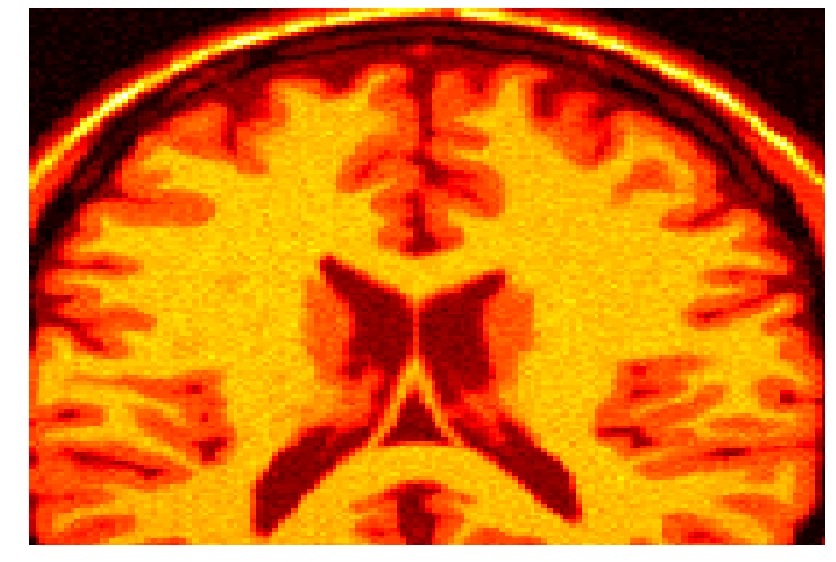}
\hfill \includegraphics[width=6.2cm]{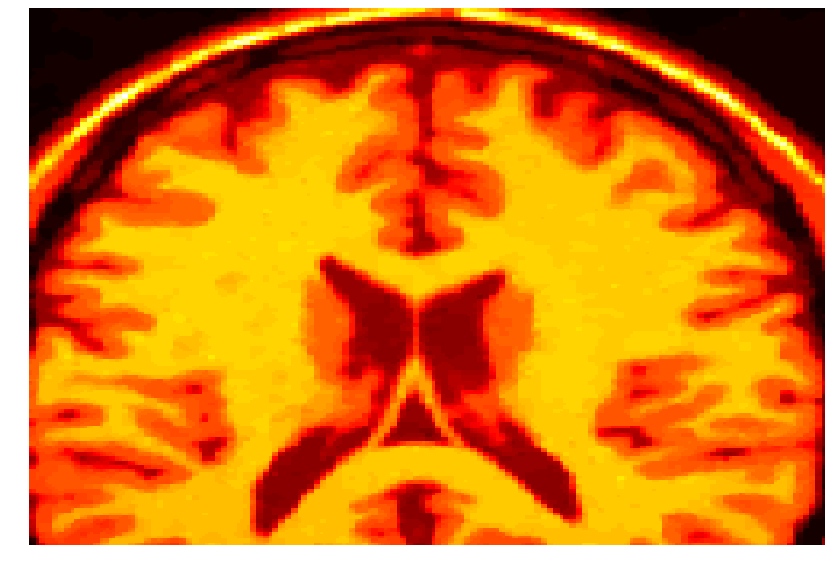} \\
\includegraphics[width=6.2cm]{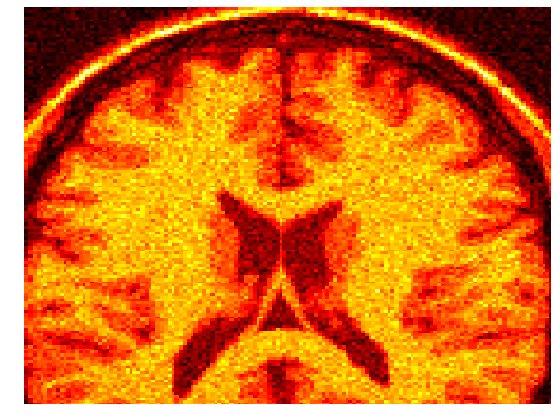} \hfill \includegraphics[width=6.2cm]{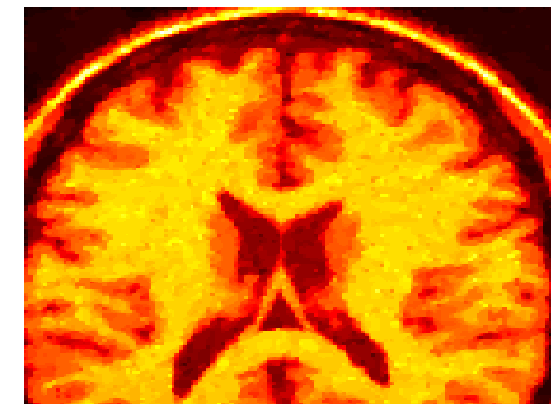} 
\end{center}
\caption{Image with 3\% noise (upper left) and its correspondent optimal denoised one (upper right); noisy image with 9\% noise (lower left) and optimal denoised image (lower right)}
\end{figure}
With this values, the optimal parameter value for the MRI image with 3\% of noise was $\lambda^*=64.1448$. When the noise in the image was of 9\%, the computed optimal weight was $\lambda^*=26.7110$.



\subsection{Training objective- multiple parameters (Gauss+Poisson)}
The importance of controlling the regularization in a denoising
problem becomes clear once two different noise distributions, modeled by fidelity terms weighted by non-negative parameters $\lambda_1$ and $\lambda_2$, are present in an image. In particular, in the following experiment we shall solve the optimization problem
$$
\min_{\lambda \geq 0}  ~\frac{1}{2} \|u - u_{o}\|^2_{L^2} + \beta \sum_{i=1}^2 \|\lambda_i\|^2
$$
subject to: 
\begin{equation}\label{gausspoissonorg}
\min_{u\geq 0} \left\{\frac{\varepsilon}{2} \|Du\|^2_{L^2} + |Du|(\Omega) + \frac{\lambda_1}{2} \|u-u_{n}\|^2_{L^2} + \lambda_2 \int_\Omega (u-u_{n}\log u)~ dx\right\}.
\end{equation}
For the characterization of a minimizer of \eqref{gausspoissonorg} we (formally) get the following Euler-Lagrange equation 
\begin{align*}
& -\epsilon\Delta u - \mathrm{div}\left(\frac{\gamma\nabla u}{\max(\gamma|\nabla u|, 1)} \right) +\lambda_1 (u-u_n) + \lambda_2 (1-\frac{u_n}{u})  -\alpha = 0\\
& \alpha\cdot u = 0,
\end{align*}
with non-negative Lagrange multiplier $\alpha \in L^2(\Omega)$, cf. \cite{HintermuellerKunisch2009}. As in \cite{SBMB09} we multiply the first equation with $u$ and get
$$
u\cdot \left(-\epsilon\Delta u - \mathrm{div}\left(\frac{\gamma\nabla u}{\max(\gamma|\nabla u|, 1)} \right) +\lambda_1 (u-u_n) \right) + \lambda_2(u-u_n) = 0,
$$
where we have used the complementarity condition $\alpha\cdot u =0$. Next, the solution $u$ is computed iteratively by using a Newton type method.

For an appropriate initial guess $u^0$, an iteration of the semismooth Newton method consists in solving the system
\begin{align}
\delta_u \left(-\epsilon \Delta u - \mathrm{div} q +\lambda_1 (u -u_n) \right) &+u \left(-\epsilon \Delta \delta_u - \mathrm{div} \delta_q +\lambda_1 \delta_u \right)\\ \nonumber + \lambda_2 \delta_u = - &u \left(-\epsilon \Delta u - \mathrm{div} q +\lambda_1 (u -u_n) \right)- \lambda_2(u-u_n),\\
\delta_q - \frac{\gamma \nabla \delta_u}{\max(1,\gamma|\nabla u|)}+ \chi_{\mathcal A_\gamma}  \gamma^2 & \frac{\nabla u^T \nabla \delta_u}{\max(1,\gamma|\nabla u|)^2} \frac{\nabla u}{|\nabla u|}=-q+ \frac{\gamma \nabla u}{\max(1,\gamma|\nabla u|)}, \label{eq:numerics gauss posson ssn iter 2}
\end{align}
for the increments $\delta_u$ and $\delta_q$. In equation \eqref{eq:numerics gauss posson ssn iter 2}, $\chi_{\mathcal A_\gamma}$ stands for the indicator function of the active set $\mathcal A_\gamma := \{ x \in \Omega: \gamma |\nabla u(x)| \geq 1 \}$.

Similarly to \cite{Delosreyes2011}, we consider a modification of the iteration based on the properties of the solution to \eqref{gausspoissonorg}. Specifically, noting that $q=\frac{\nabla u}{|\nabla u|}$ on the final active set and that $|q| \leq 1$, we replace the term $\frac{\nabla u}{|\nabla u|}$ by $\frac{q}{\max(1, \gamma |\nabla u|)}$ on the left hand side of the iteration system. The resulting iteration is then given by:
\begin{align}
\delta_u \left(-\epsilon \Delta u - \mathrm{div} q +\lambda_1 (u -u_n) \right) &+u \left(-\epsilon \Delta \delta_u - \mathrm{div} \delta_q +\lambda_1 \delta_u \right)\\ \nonumber + \lambda_2 \delta_u = - &u \left(-\epsilon \Delta u - \mathrm{div} q +\lambda_1 (u -u_n) \right)- \lambda_2(u-u_n),\\
\delta_q - \frac{\gamma \nabla \delta_u}{\max(1,\gamma|\nabla u|)}+ \chi_{\mathcal A_\gamma}   & \gamma^2 \frac{\nabla u^T \nabla \delta_u}{\max(1,\gamma|\nabla u|)^2} \frac{q}{\max(1,|q|)}=-q+ \frac{\gamma \nabla u}{\max(1,\gamma|\nabla u|)}, \label{eq:numerics gauss posson modif ssn iter 2}
\end{align}
The resulting algorithm exhibits global and local superlinear
convergence properties. In Figure \ref{fig:residuumpoissongauss} the
residuum of the algorithm in the last 4 iterations is
depicted. The parameter values used are $\mu=1e-4, \gamma=50, \beta=1e-10, h=1/60, \lambda_1=769.2199,
\lambda_2=30.6396$. From the behavior of the
residdum, local superlinear convergence is inferred. 
\begin{figure} \label{fig:residuumpoissongauss}
\begin{center}
\includegraphics[height=4.5cm,width=5.5cm]{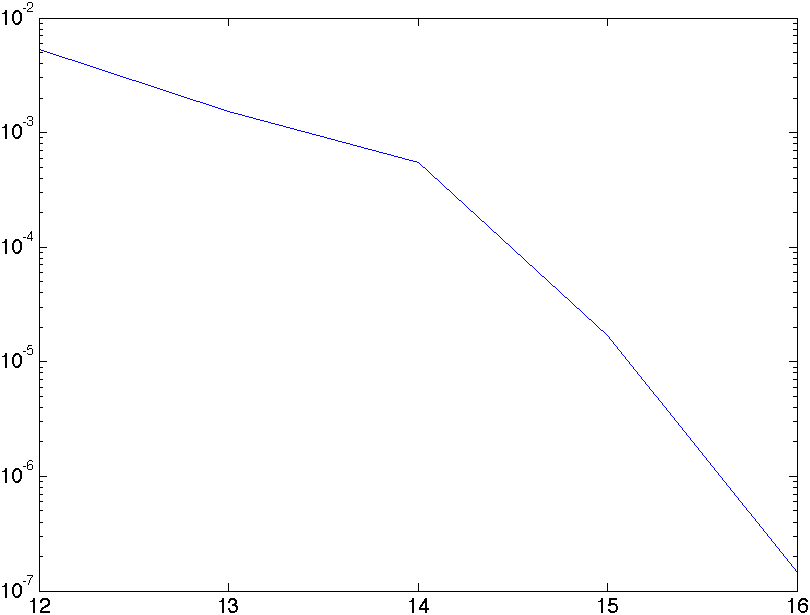} 
\end{center}
\caption{SSN residuum in the last 4
  iterations. $\mu=1e-4, \gamma=50, \beta=1e-10, h=1/60, u_1=769.2199,
u_2=30.6396$}
\end{figure}
In combination with the outer BFGS iteration, a competitive algorithm for the solution of the bilevel problem is obtained.

For the computational tests we consider the noisy zoomed image of a plane's wing (see Figure 4.4). Choosing the parameter values $\beta=1e-10, ~ \gamma =100$ and $\epsilon=1e-15$, the optimal weights $\lambda_1^*=1847.75$ and $\lambda_2^*=73.45$ were computed on a grid with mesh size step $h=1/200$. From Figure \ref{fig:experimentpoissongauss} also a good match between the original and the denoised images can be observed. The noise appears to be succesfully removed with the computed optimal weights.
\begin{figure} \label{fig:experimentpoissongauss}
\begin{center}
\includegraphics[width=4.2cm]{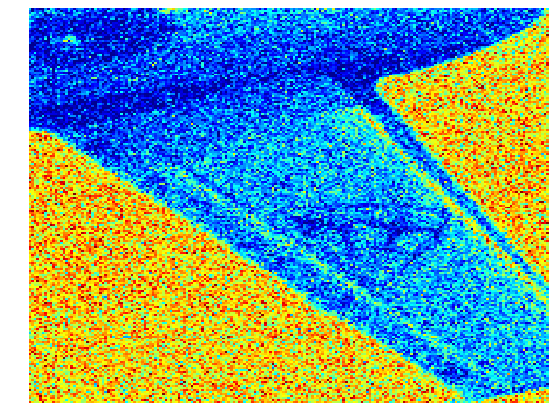} \hfill \includegraphics[width=4.2cm]{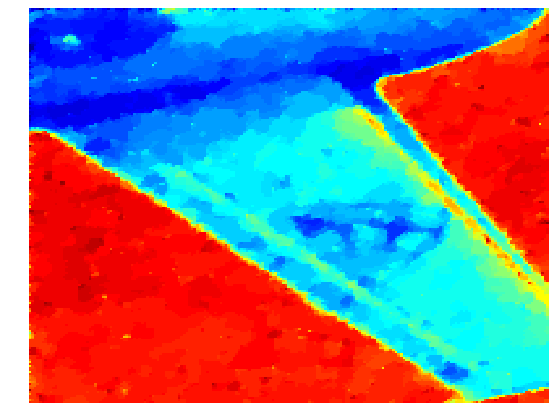} \hfill \includegraphics[width=4.2cm]{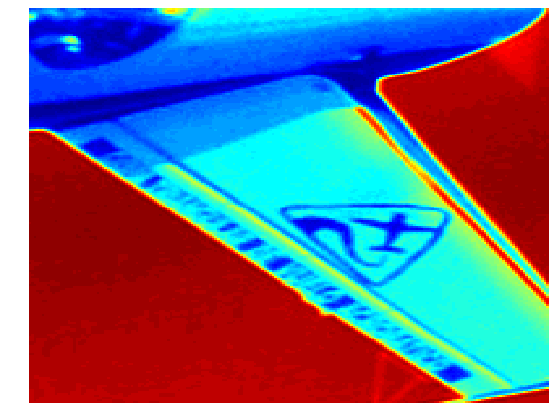}
\end{center}
\caption{Noisy (left), denoised (center) and original (right) images}
\end{figure}

Further, we tested the stability of the optimal parameter values with respect to changes in the source noisy image. This behavior is registered in Table \ref{table: sensitivity}. On the first data column the optimal weights corresponding to a zoomed image displaced by 10 grid points down and 10 grid points left is registered. The same kind of data is registered in the second column for a displacement of 5 grid points down and 5 left. In the fourth column the data corresponds to a displacement in the zoom by 5 grid points right and 5 up. Similarly, in the fifth column for 10 grid points. From Table \ref{table: sensitivity} a robust behavior of the parameter values can be inferred. Indeed, by changing the source image by 10\%, the optimal weights change less than 16\%.
\begin{table}
\centering
\begin{tabular}{|c|c|c|c|c|c|} \hline
$\hspace{.2cm}$ \%Displacement $\hspace{.2cm}$ & \hspace{.2cm} -10 \hspace{.2cm} &
\hspace{.2cm} -5 \hspace{.2cm} & \hspace{.2cm} 0 \hspace{.2cm} & \hspace{.2cm}
5 \hspace{.2cm} &\hspace{.2cm} 10 \hspace{.2cm} \\
\hline \% Sensitivity of $\lambda^*$ &15.56 &0.3 &0 &0.15 &8.26\\
\hline $\lambda_1^*$ &1270.9  &1103.4  & 1100.7 &1099.8 &1189.8\\
\hline $\lambda_2^*$ &26.39 &48.24 &46.37 &44.92 &28.13\\
\hline
\end{tabular}
\vspace{0.2cm}\caption{Optimal weight sensitivity by moving the sample image along the diagonal; $\mu=1e-12$, $\gamma=50$, $\beta=1e-10$ $h=1/100$.} \label{table: sensitivity}
\end{table}

\subsection{Impulse noise}
For the last experiment we consider images with so-called impulse noise. Specifically, we aim to solve the following parameter
estimation problem:
\begin{equation}
\min  ~\frac{1}{2} \|u - u_{o}\|^2_{L^2} + \beta \lambda^2 
\end{equation}
subject to:
\begin{multline} \label{eq:impulseVI}
\varepsilon (Du,D(v-u))_{L^2} +\lambda \int_{\Omega} |v-u_n| ~dx-\lambda \int_{\Omega} |u-u_n| ~dx\\+ \int_{\Omega} |Dv| ~dx-\int_{\Omega}  |Du| ~dx \geq 0, \forall v \in H_0^1(\Omega).
\end{multline}
Equation \eqref{eq:impulseVI} corresponds to the necessary and
sufficient optimality condition for the optimization problem:
\begin{equation}
\min_{u} \left\{\frac{\varepsilon}{2} \|Du\|^2_{L^2} + |Du|(\Omega) +
  \lambda \int_\Omega |u-u_{n} | ~dx \right\}.
\end{equation}
The $L^1$-norm is introduced to deal with the sparse impulse noise in
the image. The presence of this norm adds, however, an additional
nondifferentiability to the optimization problem.

For the numerical solution of the lower level problem we consider a
Huber type regularization of both the TV term and the
$L^1$-norm. Using a common regularization parameter $\gamma$, the
resulting nonlinear PDE takes the following form:
\begin{equation} \label{eq: impulse PDE}
-\epsilon\Delta u - \mathrm{div}\left(\frac{\gamma\nabla u}{\max(\gamma|\nabla u|, 1)}\right)  +\lambda \frac{\gamma (u-u_n)}{\max(1, \gamma |u-u_n|)}  = 0,
\end{equation}
or, in primal-dual form,
\begin{align}
-& \epsilon\Delta u - \mathrm{div}~q  +\lambda ~p  = 0,\\
& q=\frac{\gamma\nabla u}{\max(\gamma|\nabla u|, 1)},\\
& p=\frac{\gamma (u-u_n)}{\max(1, \gamma |u-u_n|)}.
\end{align}

The nonlinearities in equation \eqref{eq: impulse PDE} are present both in the quasilinear
and the semilinear terms. Both of them have to be be carefully treated
in order to obtain a convergent numerical method for the solution.

Proceeding in a similar manner as in Section 5.2, a semismooth Newton iteration for the impulse noise lower level problem is given by
\begin{align}
- \epsilon \Delta \delta_u - \mathrm{div} \delta_q +\lambda \delta_p =& -  \left(-\epsilon \Delta u - \mathrm{div} q +\lambda p \right),\\
\delta_q - \frac{\gamma \nabla \delta_u}{\max(1,\gamma|\nabla u|)}+ & \chi_{\mathcal A_\gamma} \frac{\gamma^2 \nabla u^T \nabla \delta_u}{\max(1,\gamma|\nabla u|)^2} \frac{\nabla u}{|\nabla u|}=-q+ \frac{\gamma \nabla u}{\max(1,\gamma|\nabla u|)}, \label{eq:numerics impulse ssn iter 2}\\
\delta_p - \frac{\gamma \delta_u}{\max(1,\gamma| u-u_n|)} &+  \chi_{\mathcal S_\gamma}  \frac{\gamma^2  (u-u_n) \delta_u}{\max(1,\gamma|u-u_n|)^2} \frac{(u-u_n)}{|u-u_n|} \nonumber \\ & \hspace{3cm}=-p+ \frac{\gamma (u-u_n)}{\max(1,\gamma| u-u_n|)}.
\end{align}

Using a similar argumentation as for the Gauss+Poisson noise case, we consider the modified system
\begin{align}
-& \epsilon \Delta \delta_u - \mathrm{div} \delta_q +\lambda \delta_p = -  \left(-\epsilon \Delta u - \mathrm{div} q +\lambda p \right), \label{eq:numerics impulse modif ssn iter 1}\\
&\delta_q - \frac{\gamma \nabla \delta_u}{\max(1,\gamma|\nabla u|)}+ \chi_{\mathcal A_\gamma} \gamma^2 \frac{\nabla u^T \nabla \delta_u}{\max(1,\gamma|\nabla u|)^2} \frac{q}{\max(1,|q|)}\nonumber \\ & \hspace{6cm} =-q+ \frac{\gamma \nabla u}{\max(1,\gamma|\nabla u|)}, \label{eq:numerics impulse modif ssn iter 2}\\
& \delta_p - \frac{\gamma \delta_u}{\max(1,\gamma| u-u_n|)}+
\chi_{\mathcal S_\gamma}  \frac{\gamma^2  (u-u_n)
  \delta_u}{\max(1,\gamma|u-u_n|)^2} \frac{p}{\max(1,|p|)} \nonumber
\\ & \hspace{6cm}  =-p+ \frac{\gamma (u-u_n)}{\max(1,\gamma|
  u-u_n|)}.  \label{eq:numerics impulse modif ssn iter 3}
\end{align}
where we replaced the terms $\frac{\nabla u}{|\nabla u|}$ and $\frac{(u-u_n)}{|u-u_n|}$ on the left hand side by $\frac{q}{\max(1,|q|)}$ and $\frac{p}{\max(1,|p|)}$, respectively.

The behavior of the resulting BFGS-SSN algortihm is registered in
Table \ref{table: impulse}. For the parameter values $\mu=1e-12$,
$\gamma=50$, $\beta=1e-10$ $h=1/40$ the algorithm takes 12 iterations
to converge. The number of iterations of the lower level algorithm,
given through  \eqref{eq:numerics impulse modif ssn iter 1}-
\eqref{eq:numerics impulse modif ssn iter 3}, is registered in the
last column, from which the fast convergence of the method is
experimentally verified.
\begin{table}
\centering
\begin{tabular}{|c|c|c|c|c|} \hline
$\hspace{.2cm}$ Iteration $\hspace{.2cm}$ & \hspace{.2cm} $\lambda^*$ \hspace{.2cm} & \hspace{.2cm}
Cost functional \hspace{.2cm} & \hspace{.2cm} Residuum \hspace{.2cm} &
\hspace{.2cm} \#SSN iterations \hspace{.2cm} \\
\hline 1 &10.0015 &0.0124 &0.0015 &15 \\
\hline 2 &16.7752  &0.0124 &0.0015  &3 \\
\hline 3 &19.2223 &0.0064 &4.024e-4 &14 \\
\hline 4 &10.0854 &0.0048 &5.496e-4 &12\\
\hline 5 &24.9562 &0.0123 &0.0014 &17\\
\hline 6 &26.2300 &0.0018 &1.124e-4 &18\\
\hline 7 &30.2286 &0.0017 &8.530e-5 &10\\
\hline 8 &45.8756 &0.0013 &6.794e-5 &12\\
\hline 9 &48.7340 &7.83e-4 &1.049e-5 &13\\
\hline 10 &73.2269 &7.55e-4 &9.397e-6 &7\\
\hline 11 &57.9833 &8.12e-4 &1.54e-5 &12\\
\hline 12 &58.2922 &6.81e-4 &3.20e-7 &19\\
\hline
\end{tabular}
\vspace{0.2cm}\caption{Optimal weight sensitivity by moving the sample image along the diagonal; $\mu=1e-12$, $\gamma=50$, $\beta=1e-10$ $h=1/40$.} \label{table: impulse}
\end{table}

\begin{figure} \label{fig:experimentimpulse}
\begin{center}
\includegraphics[width=4.2cm]{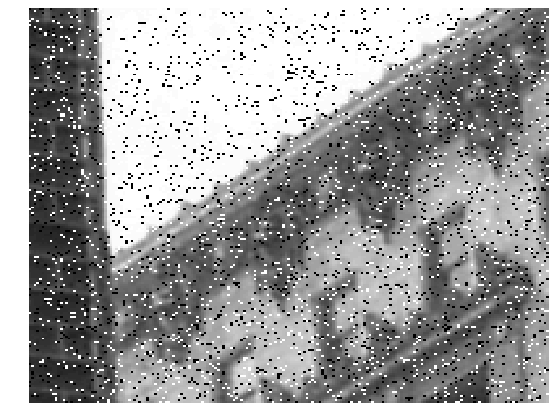} \hfill \includegraphics[width=4.2cm]{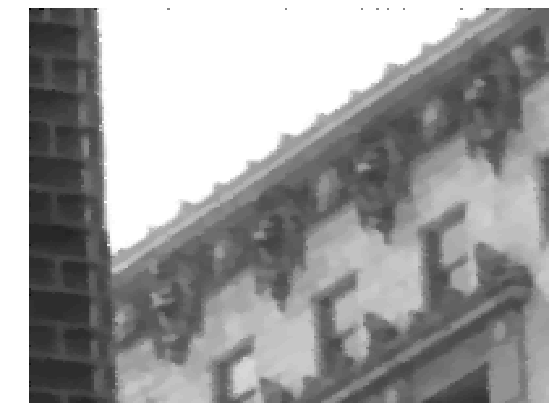} \hfill \includegraphics[width=4.2cm]{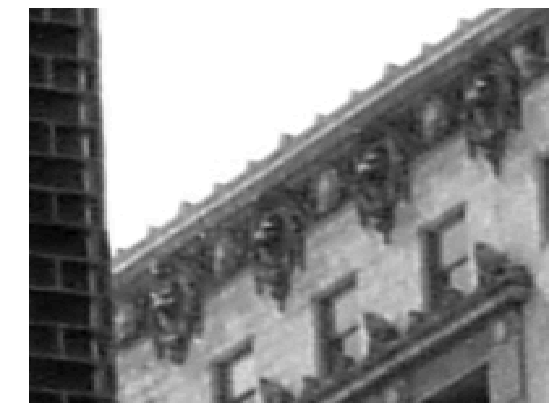}
\end{center}
\caption{Noised (left), de-noised (center) and original (right) images, $\gamma=50,~\epsilon=1e-12,~h=1/200$}
\end{figure}


\begin{thebibliography}{10}

\bibitem{ABCH08} A. Almansa, C. Ballester, V. Caselles, and G. Haro, \emph{A TV based restoration model with local constraints}, J. Sci. Comput., 34(3), 209--236, 2008.

\bibitem{AuAu08} G. Aubert, and J.-F. Aujol, \emph{A Variational Approach to remove Multiplicative Noise}, SIAM Journal on Applied Mathematics, volume 68, number 4, 925--946, January 2008.

\bibitem{AubVes} {\sc G. Aubert and L. Vese}, {\em A variational method in image recovery}, SIAM J. Numer. Anal. 34 (1997) pp. 1948--1979.

\bibitem{BCRS03} M. Bertalmio, V. Caselles, B. Roug\'e, and A. Sol\'e, \emph{TV based image restoration with local constraints}, Journal of Scientific Computing, 19:95--122, 2003.

\bibitem{Bo00} A. Bovik, \emph{Handbook of Image and Video Processing}. Academic Press, 2000.

  Maass. \emph{An optimal control problem in medical image processing}. Systems, Control, Modeling and Optimization
  Proceedings of the 22nd IFIP TC7 Conference held from July 18-22,
  2005, in Turin, Italy, 2005.
  
  
\bibitem{BBB95} Bouchitt\'e G, Braides A, Buttazzo G (1995) \emph{Relaxation results for some free discontinuity problems}. J Reine Angew Math 458:1Ð18 5.	

\bibitem{BB90} Bouchitt\'e G, Buttazzo G (1990) \emph{New lower semicontinuity results for nonconvex functionals defined on measures}. Nonlinear Anal TMA 15(7):679Ð692 6.

\bibitem{BB92}	Bouchitt\'e G, Buttazzo G (1992) \emph{Integral-representation of nonconvex functionals defined on measures}. Ann Inst H Poincar\'e 9(1):101Ð117 7.	

\bibitem{BB93} Bouchitt\'e G, Buttazzo G (1993) \emph{Relaxation for a class of nonconvex functionals defined on measures}. Ann Inst H Poincar\'e 10(3):345Ð361  

\bibitem{CasasFer1993} Casas E., Fern\'andez L. \emph{Distributed Control of Systems Governed by a General Class of Quasilinear Elliptic Equations}. J. Differential Equations, 104 (1993), pp. 20--47.

\bibitem{Cha} {\sc A. Chambolle},  {\em An algorithm for total variation minimization and applications}.  J. Math. Imaging Vision, 20 (2004),  pp. 89--97.  

\bibitem{ChaLio} {\sc A. Chambolle and P.-L. Lions}, {\em Image recovery via total variation minimization and related problems.}, Numer. Math., 76 (1997), pp. 167--188.

\bibitem{chcacrnopo10}
{\sc A. Chambolle, V. Caselles, D. Cremers, M. Novaga, and T. Pock}, {\em An Introduction to Total Variation for Image Analysis},
  {T}heoretical {F}oundations and {N}umerical {M}ethods for {S}parse {R}ecovery
  ({M}. {F}ornasier, ed.), {R}adon {S}eries on {C}omputational and {A}pplied
  {M}athematics, {D}e {G}ruyter {V}erlag, 2010, pp. 263--340.
  
\bibitem{CE05} T. F. Chan, and S. Esedoglu, \emph{Aspects of total variation regularised $L^1$ function approximation},  Siam J. Appl. Math., Vol. 65, No. 5, pp. 1817Ð1837, 2005.
  
\bibitem{CS05} T. F. Chan, and J. J. Shen, \emph{Image Processing and Analysis - Variational, PDE, wavelet, and stochastic methods}. SIAM, (2005).

\bibitem{DalMa93} G. Dal Maso, {\it An introduction to Gamma-convergence}, Birkh\"auser, Boston, 1993.

\bibitem{Delosreyes2009}
J.C. De~Los~Reyes.
\newblock Optimal control of a class of variational inequalities of the second
  kind.
\newblock {\em SIAM Journal on Control and Optimization}, Vol. 49, 1629-1658, 2011.

\bibitem{Delosreyes2011}
J.C. De~Los~Reyes.
\newblock Optimization of mixed variational inequalities arising in flow of viscoplastic materials.
\newblock {\em Computational Optimization and Applications}, DOI: 10.1007/s10589-011-9435-x, 2011.

\bibitem{DeTe84} Demengel F, Temam R (1984) \emph{Convex functions of a measure and applications}. Indiana Univ Math J 33:673Ð709.

  \bibitem{DobVog}
{\sc D. C. Dobson and C. R. Vogel}, {\em Convergence of an iterative method for
  total variation denoising}, SIAM J. Numer. Anal. 34 (1997), pp. 1779--1791.
  
  
\bibitem{DAG09} V. Duval, J.-F. Aujol, and Y. Gousseau, \emph{The TVL1 model: a geometric point of view}, SIAM Journal on Multiscale Modeling and Simulation, volume 8, number 1, 154--189, November 2009.

\bibitem{FMM12a} K. Frick, P. Marnitz, A. Munk, \emph{Statistical Multiresolution Dantzig Estimation in Imaging: Fundamental Concepts and Algorithmic Framework}, Electron. J. Stat., 6, 231--268, 2012.

\bibitem{FMM12b} K. Frick, P. Marnitz, A. Munk, \emph{Shape Constrained Regularisation by Statistical Multiresolution for Inverse Problems}, Inverse Problems, 28, 065006, 2012.

\bibitem{FMM12c} K. Frick, P. Marnitz, A. Munk, \emph{Statistical Multiresolution Estimation for Variational Imaging: With an Application in Poisson-Biophotonics}, J. Math. Imaging Vision. To appear.

\bibitem{HDR11} M. Hinterm\"uller, Y. Dong, M.M. Rincon-Camacho, \emph{Automated Regularization Parameter Selection in Multi-Scale Total Variation Models for Image Restoration}, Journal of Mathematical Imaging and Vision 40 (1), pp. 82--104, 2011.

\bibitem{HintermuellerKunisch2004}
M. Hinterm{\"u}ller and K. Kunisch
\newblock {Total bounded variation regularization as a bilaterally
              constrained optimization problem}.
\newblock {\em SIAM Journal on Applied Mathematics}, Vol. {64}, 1311--1333, {2004}.

\bibitem{HintermuellerKunisch2009} M. Hinterm{\"u}ller and K. Kunisch, \emph{Stationary Optimal Control Problems with Pointwise State Constraints}, Lecture Notes in Computational Science and Engineering, 72, 2009.
      
\bibitem{HintermuellerStadler} 
M. Hinterm{\"u}ller and G. Stadler
\newblock {An Infeasible Primal-Dual Algorithm for Total Bounded Variation--Based Inf-Convolution-Type Image Restoration}
\newblock {\em SIAM Journal on Scientific Computing}, Vol. {28}, 1--23, {2006}.

\bibitem{LCA07} T. Le, R.  Chartrand, and T.J. Asaki, \emph{A variational approach to reconstructing images corrupted by Poisson noise}, J. Math. Imaging Vision 27(3), 257--263, 2007.

 \bibitem{LLZS10} Risheng Liu, Zhouchen Lin, Wei Zhang and Zhixun Su, \emph{Learning PDEs for Image Restoration via Optimal Control}, ECCV 2010.
 
\bibitem{Morozov} V.A. Morozov, \emph{Regularization Methods for Ill--posed Problems}, CRC Press, Boca Raton, 1993. 
 
\bibitem{Nik04} M. Nikolova, \emph{A variational approach to remove outliers and impulse noise}, JMIV, vol. 20, 99--120, 2004.

\bibitem{SBMB09} A. Sawatzky, C. Brune, J. M\"uller, M. Burger, \emph{Total Variation Processing of Images with Poisson Statistics}, Proceedings of the 13th International Conference on Computer Analysis of Images and Patterns, Volume 5702, 533--540, July 2009.

\bibitem{SC09} F. R. Schmidt, D. Cremers, \emph{A Closed-Form Solution for Image Sequence Segmentation with Dynamical Shape Priors}, In Pattern Recognition (Proc. DAGM), 2009.

\bibitem{SAC05} D. Strong, J.-F. Aujol, and T. Chan, \emph{Scale recognition, regularization parameter selection, and MeyerÕs G norm in total variation regularization}, Technical report, UCLA, 2005.

\bibitem{TJFVD10} I. Tosic, I. Jovanovic, P. Frossard, M. Vetterli and N. Duric, \emph{Ultrasound Tomography with Learned Dictionaries}, IEEE International Conference on Acoustics, Speech, and Signal Processing, Dallas, Texas, International Conference on Acoustics Speech and Signal Processing ICASSP, 2010.

\bibitem{Vese} {\sc L. Vese}, {\em A study in the BV space of a denoising-deblurring variational problem}, Appl Math Optim 44 (2001), pp. 131--161.

\bibitem{VO96} C.~R. Vogel and M.~E. Oman, \emph{Iterative methods for total variation denoising}, SIAM J. Sci. Comput. \textbf{17} (1996), no.~1, 227--238, Special issue on iterative methods in numerical linear algebra (Breckenridge, CO, 1994).

\bibitem{YP03} A. M. Yip, and F. Park, \emph{Solution Dynamics, Causality, and Critical Behavior of the Regularization Parameter in Total Variation Denoising Problems}, CAM reports 03-59, 2003.

\end{thebibliography}
\end{document}